\documentclass[11pt]{article}
\usepackage{amsthm}
\usepackage{amssymb}
\usepackage{epsfig}
\usepackage{amsthm}
\usepackage{srcltx}
\usepackage[dvipsnames,usenames]{color}
\usepackage{paralist}
\usepackage{amsmath}
\usepackage{amsfonts}
\usepackage{hyperref}

\usepackage{float}

\newcounter{contador}
\newcounter{teoA}

\newtheorem{theorem}{Theorem}

\newtheorem{propo}[contador]{Proposition}

\newtheorem{lem}[contador]{Lemma}

\newtheorem{corol}[contador]{Corollary}
\theoremstyle{remark}
\newtheorem{nota}[contador]{Remark}

\newcounter{ex}

\newcommand{\sign}{{\rm sign}}

\newcommand*{\bigchi}{\mbox{\Large$\chi$}}
\newcommand{\dr}{{\rm{{d}}}}

\newcommand{\R}{{\mathbb R}}
\newcommand{\C}{{\mathbb C}}

\title{Stability index of linear random dynamical systems\footnote{The authors are supported by
Ministry of Science and Innovation--State Research Agency
of the Spanish
Government through grants PID2019-104658GB-I00 (MICINN/AEI, first and second
authors) and DPI2016-77407-P
 (MICINN/AEI/FEDER, UE, third author). The first and second authors are also supported
 by the grant 2017-SGR-1617  from
AGAUR,  Generalitat de Catalunya. The third author acknowledges the
group's research recognition 2017-SGR-388 from AGAUR, Generalitat de
Catalunya.}}

\author{Anna Cima$^{(1)}$, Armengol Gasull$^{(1,2)}$ and V\'{\i}ctor Ma\~{n}osa$^{(3)}$
  \\*[.1truecm]
{\small \textsl{$^{(1)}$Departament de Matem\`{a}tiques, Facultat
de Ci\`{e}ncies,}}
\\*[-.25truecm] {\small \textsl{Universitat Aut\`{o}noma de Barcelona,}}
\\*[-.25truecm] {\small \textsl{08193 Cerdanyola del Vall\`es, Barcelona, Spain}}
    \\*[-.25truecm] {\small \textsl{cima@mat.uab.cat,
gasull@mat.uab.cat}}\\
 \\*[-.25truecm] {\small \textsl{$^{(2)}$Centre de Recerca Matem\`{a}tica,}}
 \\*[-.25truecm] {\small \textsl{Edifici Cc, Campus de
Bellaterra,}}
    \\*[-.25truecm] {\small \textsl{08193 Cerdanyola del Vall\`es, Barcelona, Spain}}\\
\\*[-.25truecm] {\small \textsl{$^{(3)}$Departament de Matem\`{a}tiques,}}
\\*[-.25truecm] {\small \textsl{Universitat Polit\`{e}cnica de Catalunya}}
\\*[-.25truecm] {\small \textsl{Colom 1, 08222 Terrassa, Spain}}
\\*[-.25truecm] {\small \textsl{victor.manosa@upc.edu}}}

\begin{document}

\maketitle
\begin{abstract}  Given a homogeneous  linear discrete or continuous dynamical system, its stability index
is given by the dimension of the stable manifold of the zero
solution. In particular, for the $n$ dimensional case, the zero
solution is globally asymptotically stable if and only if this
stability index is $n.$ Fixed $n,$ let $X$ be the random variable that assigns
to each linear random dynamical system its stability index, and let 
$p_k$ with  $k=0,1,\ldots,n,$ denote
the probabilities  that $P(X=k)$. In this paper we obtain either the exact values  $p_k,$ or
their estimations  by combining the Monte Carlo method with a least
square approach that uses some affine relations among the values
$p_k,k=0,1,\ldots,n.$  The particular case of $n$-order homogeneous
linear random differential or difference equations is also studied
in detail.
    \end{abstract}




\section{Introduction}

Nowadays it is unnecessary to emphasize the importance of ordinary differential equations
and discrete dynamical systems to model real world phenomena, from physics to biology, from
economics to sociology. These dynamical systems, a concept that includes both continuous and
discrete models (and even dynamic equations in time-scales), can have undetermined coefficients that in the
case of real applications, must be adjusted to fixed values  that serve to make good predictions:
this is known as the identification process. Once these coefficients are fixed we obtain a deterministic model.

In recent years some authors have highlighted the utility of considering random rather than
deterministic coefficients to incorporate effects due to errors in the identification process,
natural variability in some of the physical parameters, or as a method to treat and to incorporate
 uncertainties in the model, see  \cite{CS10,CS13,SC09} for examples coming from biological
  modeling and \cite{GS} for examples coming from mechanical systems.

In the same aim that inspires some works like \cite{AL,CGM20,LPP}, in this
paper we focus on giving a statistical measure of the stability for
both discrete and continuous linear dynamical systems,
\begin{equation}\label{e:dynsys1}
\dot{\mathbf{x}}=A\,\mathbf{x} \quad\mbox{or}\quad
\mathbf{x}_{k+1}=A\,\mathbf{x}_k,
\end{equation}
 where both $\mathbf{x},\mathbf{x}_k\in\R^n$ and
$A$ is an $n\times n$ real matrix.

More concretely, in the continuous (resp. discrete) case we define
the \emph{stability index} of the origin, $s(A)$, as the number of
eigenvalues, taking into account their multiplicities, of $A$ with
negative real part (resp. modulus smaller than~1). This index
coincides with the dimension of the invariant stable manifold of the
origin. Notice also that if  $s(A)=n$ (resp. $s(A)=0$) the origin is
a global stable attractor (resp. a global unstable repeller).

 In this work we  study the probabilities $p_k$
   for a linear dynamic system \eqref{e:dynsys1}  to have a given
   stability index $k$ when the parameters of the matrix $A$ are random
   variables with a given natural distribution. As we will see in Section \ref{s:probabilityspace},  this
   distribution must be that all the elements of $A$ are
   \emph{independent and identically distributed} (i.i.d.)
   normal random variables with zero mean. We also will study the
   same question for linear $n$-th order differential equations  and for
    linear difference equations.

We also remark that our results can be extrapolated to know a measure
of the stability behaviour of critical or fixed points for general
non-linear dynamical systems, because near them they can be written
as
\begin{equation*}
\dot{\mathbf{x}}=A\,\mathbf{x}+f(\mathbf{x}),\quad\mbox{or}\quad
\mathbf{x}_{k+1}=A\,\mathbf{x}_k+f(\mathbf{x}_k),
\end{equation*}
with $f$ being  a non-linear term vanishing at zero. Moreover, while
the situation where the origin is non-hyperbolic is negligible, in
the complementary  situation, the stability index of the linear part
coincides with the dimension of the local stable manifold at the
point.

 In the continuous case, the key tool to know the stability index of a
matrix  is the Routh-Hurwitz  criterion, see for
instance~\cite[p.~1076]{GR}. This approach allows to know the number
of  roots of a polynomial with negative real part in terms of
algebraic inequalities among its coefficients. Similarly, its
counterpart for the discrete case is called the Jury criterion. It is
worth observing that in fact both are equivalent and it is possible
to get one from the other by using a M\"{o}bius transformation that
sends the left hand part of the complex plane into the complex ball
of radius 1.

In all the cases, when we do not know how to compute analytically
the searched probabilities, we introduce a two step approach to
obtain estimations of them:
\begin{itemize}

\item {\bf Step 1:} We start using the celebrated  Monte Carlo method.  Recall that
this  computational algorithm  relies on repeated random sampling
and gives estimations of the searched probabilities based on the law
of large numbers and the law of iterated logarithm, see
\cite{AG,BFS,L,Morgan}. It is the case that using $M$ samples this approach
gives the true value with an absolute error of order
$O\big(((\log\log M)/M)^{1/2}\big),$ which practically behaves as
$O(M^{-1/2}),$ where $O$ stands for the usual Landau notation. In all our simulations we will work with
$M=10^8,$ so our first approaches to the desired probabilities will
have an asymptotic absolute error of order $10^{-4}.$ More detailed  explanations of the sharpness of our
estimations for this value of $M$ are given in Section~\ref{ss:method} by using the Chebyshev inequality and the Central limit theorem.

We have used the default in-built pseudo-random number generator in the \texttt{Statistics} package of Maple in our simulations\footnote{Concretely, we use the commands \texttt{RandomVariable(Normal(0,1))} and \texttt{Sample}.}. This procedure use the Mersenne Twister method with period $2^{19937}-1$ to generate 
uniformly-distributed pseudo-random numbers, and then the Ziggurat method, which is a kind of rejection sampling algorithm,  to obtain the normally-distributed  pseudo-random numbers, see \cite{MT} and \cite{MN}.  Observe that our sample size $M=10^8$ is much smaller than the period of the pseudo-random number generator, which is greater that $10^{6001}$.

\item {\bf Step 2:} Since the results of the plain Monte Carlo simulations do not satisfy certain linear constrains concerning the true probabilities, we propose to correct them by using a least squares approach.
We take as final estimates of the true probabilities  the least
squares solution (\cite[Def.~6.1]{Schott}) of the inconsistent
overdetermined system obtained when the Monte Carlo's observed
relative frequencies are forced to satisfy these linear constrains.
See Section~\ref{ss:method} for more details. We would like to remark that there are other options to improve plain Monte Carlo simulations like variance~reduction and quasi-Monte Carlo methods \cite{AG,L}.
\end{itemize}

To have a flavour of the type of results that we will obtain  we
describe several consequences  of some of our results for  linear
homogeneous differential or difference  equations of order $n$ with
constant coefficients (see Sections~\ref{ss:LRODEn}
and~\ref{ss:LRdifference}). A first result is that in both cases the
expected stability index is $n/2.$ Moreover, let $r_n$ denote the
probability of the $0$ solution to be a global stable attractor
(stability index equals $n$) for them. Then, for differential
equations, $r_n\le1/2^n.$ Furthermore, $r_1=1/2,$ $r_2=1/4,$
$r_3=1/16$ and our two step approach gives that
$r_4\simeq0.00925,$ $r_5\simeq0.00071,$ and that $r_k$ is smaller
that $10^{-4}$ for bigger~$k.$ In the case of difference equations
we prove that $r_1=1/2$ and
$r_2=\frac1\pi\arctan(\sqrt{2})\simeq0.304.$

\section{A suitable probability space}\label{s:probabilityspace}

In our approach, the starting point is to determine which is the natural choice 
of the  probability space and the distribution law of the
coefficients of the linear dynamical system. Only after this step is
fixed we can ask for the probabilities of some dynamical features or
some phase portraits.

For completeness, we start with some previous considerations and
with an example, already considered in the literature, see
\cite{AL,LPP,Strog}. Consider the planar linear differential system:
\begin{equation}\label{e:lineal} \left(\begin{array}{c}
\dot{x}\\
\dot{y}
\end{array}\right)
=
\left(
\begin{array}{cc}A & B\\
C & D\end{array}\right)\left(\begin{array}{c}
x\\
y
\end{array}\right)\end{equation} where  $A,B,C,D$ are random variables, so
we can set the sample space to be  $\Omega=\R^4$. It is  plausible to require that these real random variables are independent and identically distributed (i.i.d.)~and continuous. Also, according to the \emph{principle of
indifference} (or principle of insufficient reason) \cite{Conrad},
it would seem reasonable to impose that these variables were such that the
random vector $(A,B,C,D)$ had some kind of uniform distribution in
$\mathbb{R}^4$. But there is no uniform distribution for unbounded
probability spaces. Nevertheless, there is a natural election for
the distribution of the variables $A,B,C$ and $D$.

Indeed, it is well-known that the phase portrait of the above system
does not vary if we multiply the right-hand side of both equations
by a positive constant  (which corresponds  to a proportional change
in the time
scale). This means that in the space of parameters, $\R^4,$ all the
systems with parameters belonging to the same half-straight line
passing through the origin are topologically equivalent and in
particular have the same stability index. Hence, we can ask for  a
probability distribution density $f$ of the coefficients such that
the random vector
\begin{equation}\label{e:randomvector}
\left(\frac{A}{S}\,,\frac{B}{S}\,,\frac{C}{S}\,,\frac{D}{S}\right),
\quad\mbox{with}\quad S=\sqrt{A^2+B^2+C^2+D^2},
\end{equation}
has a uniform distribution on the sphere ${\mathbb S}^3 \subset \R^4$.
This achieves our objective, since ${\mathbb S}^3 $  is a compact set.

The question is: which are the probability densities $f$ that give
rise to a  uniform distribution of the vector \eqref{e:randomvector}
on the sphere? The answer is that, just assuming that $f$ is
continuous and positive, $f$ must be the density of a normal random
variable with zero mean. Moreover, this result is true for arbitrary
dimension: see the next theorem.  We remark that the converse result
is well--known \cite{Mars,Mull}.

\begin{theorem}\label{t:normals}
Let $X_1,X_2,\ldots,X_n$ be i.i.d.~one-dimensional random  variables
with a continuous positive density function $f$. The
random vector
$$
\left(\dfrac{X_1}{S}, \dfrac{X_2}{S},\ldots,\dfrac{X_n}{S}\right),
\quad\mbox{with}\quad S=\Big(\sum_{i=1}^{n} X_i^2\Big)^{1/2},
$$
has a uniform distribution  in $\mathbb{S}^{n-1}\subset\mathbb{R}^n$
if and only if each $X_i$ is a normal random variable with zero
mean.
\end{theorem}

\medskip

Curiously, in the case that we cannot assign uniform distributions,
there is an extension of the indifference principle which suggests to
use those distributions that maximize the entropy, i.e. the quantity
$h(f)=-\int_\Omega f(x)\ln(f(x))\mathrm{d}x$ for any given density
$f$. The one-dimensional random variables with continuous
probability density function $f$ on $\Omega=\R$ that maximize the
entropy are again the Gaussian ones, \cite[Thm 3.2]{Conrad}.

Of course, if instead of properties concerning general dynamical
systems one focuses on particular models in which the parameters
have specific restrictions -due to physical or biological  reasons-
one must consider other type of distributions, see for
instance~\cite{SC09}.

Using Theorem \ref{t:normals}, and going back to the initial
motivating  example,  in order to study~(\ref{e:lineal}) we have to
consider the probability space $(\Omega,\mathcal{F},P)$ where
$\Omega=\R^4$, $\mathcal{F}$ is the $\sigma$-algebra generated by
the open sets of $\R^4$ and $P:\mathcal{F}\rightarrow [0,1]$ is the
probability function with density
$\frac{1}{4\pi^2}\mathrm{e}^{-(a^2+b^2+c^2+d^2)/{2}},$ where for
simplicity we take  variance~1  in each marginal density function.

For instance, assume that we want to compute the probability
$\alpha$ of system (\ref{e:lineal}) to have exactly one  eigenvalue with
negative real part. Next, we observe that the probability of having one null eigenvalue is zero. This is because the event which characterizes this possibility is a subset of an event which is itself described
by an algebraic equality between the random variables $A, B, C, D$. This subset has Lebesgue measure zero and therefore, by virtue of the fact that the joint distribution is continuous, the probability of this event, and therefore the event characterizing the null eigenvalue, must also be zero. Thus we have that $\alpha$ coincides 
with the probability of having a saddle (stability index~$1$) at the
origin, i.e. $AD-BC<0$. Then, the open set
$\mathcal{U}:=\{(a,b,c,d)\in\R^4:ad-bc<0\}$ belongs to $\mathcal{F}$
and
$$\alpha=P(AD-BC<0)\,=\,\frac{1}{4\pi^2}\int_{\mathcal{U}}
\mathrm{e}^{-\frac{a^2+b^2+c^2+d^2}{2}}\dr a\,\dr b\,\dr c\,\dr d,$$
which is $1/2$ by symmetry, as we will see.

\begin{proof}[Proof of Theorem \ref{t:normals}] Let $(X_1,\ldots,X_n)$
be the random vector associated with the random variables of the statement, with joint continuous density function
$g(x_1,\ldots,x_n)$. We claim that
\begin{equation}\label{radigen}
g(x_1,\ldots,x_n)=h(x_1^2+\cdots+x_n^2),
\end{equation} for
some continuous function $h$.

Taking spherical coordinates, we consider the new random vector
$(R,\Theta)\in\mathbb{R}^n$ where $R=(X_1^2+\cdots+X_n^2)^{1/2}$ and
$\Theta=(\Theta_1,\ldots,\Theta_{n-1}).$ We have $X_1=R \cos
\Theta_1$, $X_2=R \sin \Theta_1 \cos \Theta_2,\ldots$ $X_{n-1}=R
\sin \Theta_1 \sin \Theta_2 \cdots \sin \Theta_{n-2} \cos
\Theta_{n-1}$ and $X_n=R \sin \Theta_1 \sin \Theta_2 \cdots \sin
\Theta_{n-2} \sin \Theta_{n-1}$. By the change of variables theorem,
the joint density function of $(R,\Theta)$ is {\small
$$g_{R,\Theta}(r,\theta)=g(r\cos(\theta_1),\ldots, r\sin(\theta_1)\cdots\sin(\theta_{n-1}))
\,r^{n-1}\sin^{n-2}(\theta_1)\sin^{n-1}(\theta_2)\cdots\sin(\theta_{n-2})\cdot \bigchi$$}
where $\theta=(\theta_1,\ldots,\theta_{n-1})$, and
$$\bigchi:=
\bigchi_{[0,\infty)}(r)\cdot\,\bigchi_{[0,2\pi)}(\theta_{n-1})\,\cdot\,
\prod\limits_{i=1}^{n-2}\bigchi_{[0,\pi)}(\theta_i),$$ where
$\bigchi_{A}$  stands for the characteristic function of the set
$A.$

The density function of $(R,\Theta)$ conditioned to $R$,
$g_{\Theta|R},$ is
$$g_{\Theta|R}(r,\theta):=\frac{g_{R,\Theta}(r,\theta)}{g_{R}(r)},$$
where $g_{R}(r)$ is the marginal density of $R$:
$$g_{R}(r):=\int_{0}^{\pi}\!\cdots\!\int_{0}^{\pi}\int_{0}^{2\pi}
g(r\cos(\theta_1),\ldots,
r\sin(\theta_1)\cdots\sin(\theta_{n-1}))\,\dr \mathcal{S},$$ where
$\dr \mathcal{S}= r^{n-1}\sin^{n-2}(\theta_1)\sin^{n-1}(\theta_2)
\cdots\sin(\theta_{n-2})\dr\theta_{n-1}\cdots\dr \theta_{1}$ is the
$n$-dimensional surface element in spherical coordinates.

To prove the statement, we need to characterize which are the joint
density functions $g(x_1,\ldots,x_n)$ such that when we fix $R=r$,
the probability on the $(n-1)$-dimensional sphere of radius $r$,
denoted by $\mathbb{S}^{n-1}(r)$, is uniformly distributed. In such
a case the partial spherical segment $\Sigma_r=\{R=r,\,
\theta_i\in[\alpha_i,\beta_i]$ for $i=1,\ldots,n-1\}$ must have
probability $
P(\Sigma_r)={\mathcal{S}(\Sigma_r)}/{\mathcal{S}(\mathbb{S}^{n-1}(r))}
$ where $\mathcal{S}$ denotes the surface area. Set
$\alpha=(\alpha_1,\ldots,\alpha_{n-1})$ and
$\beta=(\beta_1,\ldots,\beta_{n-1})$. Notice that
$$
\mathcal{S}(\Sigma_r)=\int_{\alpha_1}^{\beta_1}\!\cdots\!
\int_{\alpha_{n-1}}^{\beta_{n-1}} \dr
\mathcal{S}=r^{n-1}\,A(\alpha,\beta)
$$ where
$$
A(\alpha,\beta)= \int_{\alpha_1}^{\beta_1}\!\cdots\!
\int_{\alpha_{n-1}}^{\beta_{n-1}}
\sin^{n-2}(\theta_1)\sin^{n-1}(\theta_2)\cdots\sin(\theta_{n-2})\dr\theta_{n-1}\cdots\dr
\theta_{1}$$ and $
\mathcal{S}(\mathbb{S}^{n-1}(r))=\frac{2\pi^\frac{n}{2}}{\Gamma\left(\frac{n}{2}\right)}\,r^{n-1}.
$ Hence, on the one hand,
$$
P(\Sigma_r)=\Gamma\left(\frac{n}{2}\right)\,\frac{A(\alpha,\beta)}{2\pi^\frac{n}{2}},
$$
which does not depend on $r$. On the other hand,
$$
P(\Sigma_r)=\int_{\alpha_1}^{\beta_1}\!\cdots\!
\int_{\alpha_{n-1}}^{\beta_{n-1}} g_{\Theta|R}\,\dr \theta$$ where
$\dr \theta=\dr\theta_{n-1}\cdots\dr \theta_{2}\dr \theta_{1}$. This
implies that 
{\small
$$
\int_{\alpha_1}^{\beta_1}\!\cdots\! \int_{\alpha_{n-1}}^{\beta_{n-1}}
\frac{g(r\cos(\theta_1),\ldots,
r\sin(\theta_1)\cdots\sin(\theta_{n-1}))\,r^{n-1}
\sin^{n-2}(\theta_1)\sin^{n-1}(\theta_2)\cdots\sin(\theta_{n-2})\cdot
\bigchi} {g_R(r)} \,\dr\theta$$
$$
=\frac{\Gamma\left(\frac{n}{2}\right)}{2\pi^{\frac{n}{2}}}\,
\int_{\alpha_1}^{\beta_1}\!\cdots\! \int_{\alpha_{n-1}}^{\beta_{n-1}}
\sin^{n-2}(\theta_1)\sin^{n-1}(\theta_2)\cdots\sin(\theta_{n-2})\dr\theta,
$$}
\noindent \!\!for all $\alpha_i,\beta_i\in[0,\pi)$ for $i=1,\ldots,n-2$ with
$\alpha_i<\beta_i$ and $\alpha_{n-2},\beta_{n-2}\in[0,2\pi)$ with
$\alpha_{n-2}<\beta_{n-2}$.  This last  equality implies that almost
everywhere
$$\frac{\Gamma\left(\frac{n}{2}\right)}{2\pi^{\frac{n}{2}}}=\frac{g(r\cos(\theta_1),
\ldots,
r\sin(\theta_1)\cdots\sin(\theta_{n-1}))\,r^{n-1}}{g_{R}(r)},$$ and
therefore $g(r\cos(\theta_1),\ldots,
r\sin(\theta_1)\cdots\sin(\theta_{n-1}))$ is a function that only
depends on $r.$  In consequence, writing this fact in Cartesian
coordinates, we get that almost everywhere
$g(x_1,\ldots,x_n)=h(x_1^2+\cdots+x_n^2),$ for some continuous
function $h$ and the claim~\eqref{radigen} follows.

Now we complete the proof. Since $X_1,\ldots,X_n$ are i.i.d.~with
positive density $f$, we know that $g(x_1,\cdots,x_n)=f(x_1)\ldots
f(x_n).$ So equation (\ref{radigen}) can be expressed as
$$f(x_1)\cdots f(x_n)\,=\,h(x_1^2+\cdots+x_n^2)\quad \mbox{for all}\quad (x_1,\ldots,x_n)\in\R^n$$
where $h$ is a positive function. Taking $x_2=\cdots=x_n=0$ we have that
$f(x_1)\,f(0)^{n-1}=h(x_1^2)$ and
 $h(0)=(f(0))^n>0.$ Thus,
$$f(x_1)\ldots f(x_n)=\frac{h(x_1^2)}{(f(0))^{n-1}}\cdots \frac{h(x_n^2)}{(f(0))^{n-1}}=
\frac{h(x_1^2)}{(f(0))^{n}}\cdots \frac{h(x_n^2)}{(f(0))^{n}}\,(f(0))^{n}=h(x_1^2+\cdots+x_n^2).
$$ Hence, using that $h(0)=(f(0))^n>0,$
$$\frac{h(x_1^2)}{h(0)}\cdots \frac{h(x_n^2)}{h(0)}\,=\frac{h(x_1^2+\cdots+x_n^2)}{h(0)}.
$$
Taking $H(\xi):=h(\xi)/h(0),$ and $u_i=x_i^2$, it holds that
\begin{equation}\label{expgen}
H(u_1)\cdots H(u_n)\,=\,H(u_1+\cdots+u_n)\quad \mbox{with}\quad
H(0)=1.
\end{equation}
Hence, $\varphi(U)=\log (H(u))$ is a continuous function that
satisfies the Cauchy's functional equation
$$
\varphi(u_1)+\cdots +\varphi(u_n)\,=\,\varphi(u_1+\cdots+u_n) \quad
\mbox{with}\quad \varphi(0)=0.
$$
It is well--known that all its continuous solutions are
$\varphi(x)=ax,$ for some $a\in\R.$ Hence all continuous solutions
of (\ref{expgen}) are $H(x)=\mathrm{e}^{ax}.$

As a consequence, $f(x)=b\,\mathrm{e}^{ax^2}$ for some
$(a,b)\in\R^2.$ Since $f$ is a density function, $a<0.$ Moreover,
using $\int_{-\infty}^{\infty}b\mathrm{e}^{ax^2}\dr
x=b\sqrt{-\pi/a}=1$, and setting  $a=-1/(2\sigma^2)$, we get that
$$f(x)=\frac{1}{\sqrt{2\pi \sigma^2}}\,\mathrm{e}^{-\frac{x^2}{2\sigma^2}},$$
so each variable $X_i$ is a  normal random variable $\mathrm{N}(0,\sigma^2).$

The converse part is straightforward and well--known
\cite{Mars,Mull}.
~\end{proof}

\begin{nota} The continuity condition for $f$ in Theorem \ref{t:normals} is
relevant since Equation (\ref{expgen}) also admits non-continuous
solutions that  can be constructed, for instance, from
non-continuous solutions of the Cauchy's functional equation known
for $n=2,$  see~\cite{Jones}.
\end{nota}

\section{A preliminary result and methodology}\label{s:linear}

We will investigate the probabilities of having a certain stability
index for several linear dynamical systems with random coefficients.
In particular we consider:
\begin{itemize}
\item [(a)] Differential systems  $\dot{\mathbf{x}}=A\,\mathbf{x}$  where $\mathbf{x}\in\R^n$ and
$A$ is a real constant $n\times n$ matrix,
\item [(b)] Homogeneous linear differential equation of order $n$
with constant coefficients: $a_nx^{(n)}+a_{n-1}x^{(n-1)}+\cdots+a_1
x'+a_0x=0,$
\item [(c)] Linear discrete systems  $b \,\mathbf{x}_{k+1}=A\,\mathbf{x}_k$ where
$\mathbf{x}_k\in\R^n,$ $b\in\R$; and $A$ is a real constant $n\times
n$ matrix,
\item [(d)] Linear homogeneous difference equation of order $n$ with
constant coefficients $a_nx_{k+n}+a_{n-1}x_{k+n-1}+\cdots+a_1x_{k+1}+a_0x_k=0$.
\end{itemize}

Notice that in the four situations the behaviour of the dynamical
systems does not change if we multiply all the involved constants by
the same positive real number. This fact situates the four problems
in the same context that the motivating example \eqref{e:lineal}.
Hence, following the results of Section \ref{s:probabilityspace}, in
all the cases, we may take the coefficients to be i.i.d.~random normal variables with
zero mean and variance 1.


Hence in all cases we have a well--defined probability space
$(\Omega,\mathcal{F},P)$, where $\Omega=\R^m,$ with
$m=n^2,n+1,n^2+1$ or $n+1$ according we are in case (a), (b), (c) or
(d), respectively, $\mathcal{F}$ is the $\sigma$-algebra generated
by the open sets and for each $\mathcal{A}\in\mathcal{F}$,
\begin{equation}\label{eq:pa}
P(\mathcal{A})= \frac{1}{(\sqrt{2\pi})^{m}}\,\int_{\mathcal{A}}
\mathrm{e}^{-||\mathbf{a}||^2/2}\dr \mathbf{a},
\end{equation}
where $\mathbf{a}=(a_1,a_2,\ldots,a_m),$
$||\mathbf{a}||^2=\sum_{j=1}^m a_j^2$ and $\dr\mathbf{a}=\dr
a_1\,\dr a_2\ldots\dr a_m.$ For instance the matrices $A$ appearing
in case (a) and (c) are the so called {\it random matrices}.

The use of Routh-Hurwitz algorithm is a very useful tool  to count
the number of roots  of a polynomial with negative real parts  and
it is implemented in many computer algebra systems. These conditions
are given in terms of algebraic inequalities among the coefficients
of the polynomials. Let us recall how to use it to count the number
of roots with modulus less than one of a polynomial  and, hence,  to
obtain the so  called Jury conditions.

Given any polynomial
$
Q(\lambda)=q_n\lambda^n+q_{n-1}\lambda^{n-1}+\cdots+q_1\lambda+q_0
$
with $q_j\in\C,$ by using the conformal transformation
$
\lambda=\frac{z+1}{z-1},
$
we get the associated polynomial
\begin{equation}\label{e:polyestar}
Q^\star(z)=q_n(z+1)^n+q_{n-1}(z+1)^{n-1}(z-1)+\ldots+q_0(z-1)^n.
\end{equation}
It is straightforward to observe that $\lambda_0\in
    \mathbb{C}$ is a root of of $Q(\lambda)$ such that $|\lambda_0|<1$
    if and only if $z_0=(\lambda_0+1)/(\lambda_0-1)$  is a root of
    $Q^\star(z)$ such that $\mathrm{Re}(z_0)<0$.

Hence, because  Routh-Hurwitz and Jury conditions are
 semi-algebraic,
 in all cases the random variable that $X$ that assigns to each
 dynamical system its stability index~$k, 0\le k\le n,$ is
 measurable. Hence $\mathcal{A}_k:=\{\mathbf{a}\in\R^m \,:\, X(\mathbf{a})=k\}\in\mathcal{F}$
 and its probability
 $p_k:=P(\mathcal{A}_k)$ is well--defined. Observe also that the non-hyperbolic cases
 are totally negligible because in their characterization some
 algebraic equalities appear. In this paper we will either calculate or estimate in
 the four situations the values $p_k$ for $k\le10.$

\subsection{A preliminary result}\label{ss:prelim}

In three of the above considered cases we will apply the following
auxiliary result:

\begin{lem}\label{l:lemmageneral}
Let $(\Omega,\mathcal{F},P)$ be a probability space and let
$Y:\Omega\to \mathbb{R}$ be  a discrete random variable with image
$\mathrm{Im}(Y)=\{0,1,\ldots,n\}$, and probability mass function
$p_k=P(Y=k)$ such that $p_k=p_{n-k}$ for all $k=0,\ldots,n$. Then
$E(Y)=\sum_{k=0}^n k p_k=n/2.$ Moreover
\begin{enumerate}
\item[(a)] If $n$ is odd then
$2\sum_{k=0}^{\frac{n-1}{2}} p_k=1.$ In particular, when
$n=1,$ $p_0=p_1=\frac{1}{2}.$
\item[(b)]
If $n$ is even and $n\geq 2$ then
$2\sum_{k=0}^{\frac{n}{2}-1} p_k\,+\, p_{\frac{n}{2}}=1.$
\end{enumerate}

\noindent If, additionally, $n$ is even\footnote{When $n$ is odd the
imposed equalities automatically  hold.} and $\sum_{i\,\mathrm{odd}}
p_i=\sum_{i\,\mathrm{even}} p_i=\frac{1}{2}$, then
\begin{enumerate}

\item[(c)] If $\frac{n}{2}$ is even, then
$ 2\sum_{k=0,\, k\,\mathrm{ even}}^{\frac{n}{2}-2} p_k+
p_{\frac{n}{2}}=\frac{1}{2},$ and $2\sum_{k=1,\, k
\,\mathrm{odd}}^{\frac{n}{2}-1} p_k= \frac{1}{2}.$ In particular,
when $n=4,$ $p_1=p_3=\frac{1}{4},$ $p_2=\frac{1}{2}-2p_0$ and
$p_4=p_0.$

\item[(d)] If $\frac{n}{2}$ is odd, then
$ 2\sum_{k=0,\, k\,\mathrm{ even}}^{\frac{n}{2}-1} p_k=\frac{1}{2},$
and $\sum_{k=1,\, k \,\mathrm{odd}}^{\frac{n}{2}-2} p_k+
p_{\frac{n}{2}}= \frac{1}{2}.$ In particular, when $n=2,$
$p_0=p_2=\frac{1}{4}$ and $p_1=\frac{1}{2}.$
\end{enumerate}

\end{lem}

\begin{proof} We start proving that $E(Y)=n/2.$  Assume for instance that $n$ is odd.
Since $p_k=p_{n-k},$ its holds that $kp_k+(n-k)p_{n-k}= np_k,$ for
each $k\le (n-1)/2.$ Hence,
\begin{align*}
E(Y)&=np_0+np_1+\cdots
+np_{\frac{n-1}{2}}=\frac{n}{2}\left(2p_0+2p_1+\cdots
+2p_{\frac{n-1}{2}}\right)\\
&=\frac{n}{2}\big(
(p_0+p_n)+(p_1+p_{n-1})+\cdots+(p_{\frac{n-1}{2}}+p_{\frac{n+1}{2}})\big)=\frac
n 2.
\end{align*}
When $n$ is even the proof is similar.

The proof of all the four items is straightforward and we omit
it.~\end{proof}

\subsection{Experimental methodology}\label{ss:method}

In all the cases considered in the paper, when we can not give an
exact value of the probabilities $p_k$ we start estimating them by
using the \emph{Monte Carlo} method, see \cite{Morgan}. The estimates obtained 
(namely, the observed relative frequencies) are then improved via the
  \emph{least squares} method, by using the linear
constraints given in Corollaries~\ref{c:corolsistemes},~\ref{c:propoEDOL3}
and~\ref{c:propoEDD}.

 In all the cases we will use Monte Carlo method with $M=10^8$ to obtain an estimation, say $\widetilde p,$ for
a probability $p:=P(\mathcal{A})$ like the one given in equality \eqref{eq:pa} for different measurable sets~$\mathcal
A.$  Further details for each concrete situation are given in
each of the following subsections.

In a few words, recall that $\widetilde p$ is given by
 the proportion of samples that are in $\mathcal{A}$. For studying, for a given $M,$ how close are $p$ and $\widetilde p,$  let $B_j, j=1,\ldots, M,$ i.i.d.~Bernoulli random variables, where each one of them
that takes the value $1$ with probability $p$ and the value $0$ with
probability~$1-p.$ 

Define $P_M=\frac1 M \sum_{j=1}^M B_j.$ Then, the
value obtained for the random variable $P_M,$ $\widetilde p$ is the
approximation of $p$ given by Monte Carlo method. Let us see, by using Chebyshev inequality or
the Central limit theorem, that with very high probability,
$\widetilde p$ is a good approximation of~$p.$

Notice first that $\operatorname{E}(P_M)=p$ and due to the  independence of the
$B_j,$
\[
\operatorname{Var}(P_M)=\operatorname{Var}\left( \frac1 M
\sum_{j=1}^M B_j\right)=\frac1{M^2}M
\operatorname{Var}(B_1)=\frac{p(1-p)}M\le \frac1{4M},
\]
because $p(1-p)\le1/4.$ Recall also that for each $\varepsilon>0$  and any random variable
$X,$ with $\operatorname{E}(X^2)<\infty,$ the  Chebyshev inequality
reads as
\[
P\left(|X-\operatorname{E}(X)|<\varepsilon \right)\ge
1-\frac{\operatorname{Var}(X)}{\varepsilon^2}.
\]
Hence, applying the Chebyshev inequality to $X=P_M$ we get that
\[
P\left(|P_M-p|<\varepsilon \right)\ge 1- \frac{p(1-p)}{M\varepsilon^2}\ge
1-\frac{1}{4M\varepsilon^2}.
\]
 Taking $M=10^8,$  as in our computations, denoting $\widetilde
p=P_{10^8},$ and considering $\varepsilon=10^{-3}$ we get that
\[
P\left(|\widetilde p-p|<10^{-3} \right)\ge
1-\frac1{400}=\frac{399}{400}=0.9975.
\]

Let us see,  by using the Central limit theorem, that the above probability seems to be much bigger. By this theorem we know that for $M$ big
enough, and $p(1-p)M$ also big enough, the distribution of the
random variable
\[
\frac{P_M-\operatorname{E}(P_M)}{\sqrt{\operatorname{Var}(P_M)}}
=\frac{P_M-p}{ \sqrt{\frac{p(1-p)}{M}}}
\]
can be practically considered to be a random variable $Z$ with
distribution $\mathrm{N}(0,1).$ In fact, in Statistics it is usually imposed that $p(1-p)M>18.$ Hence, unless $p$ is very close to~$0$ or~$1$, the value $M=10^{8}$ is big enough. Hence
\begin{align*}
P\left(|P_M-p|<\varepsilon \right)=&P\left(
\frac{\sqrt{M}|P_M-p|}{\sqrt{p(1-p)}}< \frac{\varepsilon
	\sqrt{M}}{\sqrt{p(1-p)}} \right)\simeq P\left(|Z|<
\frac{\varepsilon
	\sqrt{M}}{\sqrt{p(1-p)}} \right)\\
=&2\Phi\left(\frac{\varepsilon
	\sqrt{M}}{\sqrt{p(1-p)}}\right)-1>2\Phi\left(2\varepsilon
	\sqrt{M}\right)-1,
\end{align*}
where $\Phi$ is the distribution function of a $\mathrm{N}(0,1).$ Taking
again $M=10^{8}$ and $\varepsilon=10^{-3}$ or $\varepsilon=2\times 10^{-4}$  we get
\begin{align*}
&P\left(|\widetilde p-p| <10^{-3} \right) \gtrsim2\Phi\left(20\right)-1>1-10^{-88},\\
&P\left(|\widetilde p-p| <2\times10^{-4} \right) \gtrsim2\Phi\left(4\right)-1>0.99993.
\end{align*}

In fact, for instance looking at the values $p_k$ of Table 2  for $n=2$ in Section \ref{ss:LRDS}, that can also be
obtained analytically, we get that $|\widetilde p_k-p_k|\le 6\times
10^{-5},$ for $k=0,1,2.$ So, the actual bound is smaller that the bounds
obtained above.

Finally, and to illustrate how the error decays when the sample size increases, we show the evolution of the errors in one case where the true probabilities are known. We consider the second order difference equation $A_2 x_{k+2}+A_1 x_{k+1}+A_0 x_k=0$ where $A_i$ are i.i.d.~random variables with $\mathrm{N}(0,1)$ distribution. The stability index is given by the number of zeroes with modulus smaller than~$1$ of the characteristic polynomial $Q(\lambda)=A_2 \lambda^2+A_1 \lambda+A_0$. Let $X$ be the random variable that counts the number of roots with modulus smaller than $1$ of $Q(\lambda)$, and $p_k=P(X=k)$ for $k=0,1,2$. The true value of the probabilities $p_k$ is obtained in Corollary \ref{c:propoEDD}. 
Performing Monte Carlo simulations with $M=10^m$ with $m=2,\ldots,10$ we obtain the  observed frequencies
$\widetilde{p}_2(m)$ shown in Table 1. These frequencies are the estimated probabilities for the origin to be asymptotically stable. Notice that in Proposition \ref{p:propoEDiff} and in Corollary \ref{c:propoEDD} we prove that $p_0=p_2=\arctan(\sqrt{2})/\pi$ and, of course, $p_1=1-p_0-p_2={2}\arctan(1/\sqrt{2})/\pi$). For $M=10^m$ we denote the absolute error $e_m=|\widetilde{p}_2(m)-p_2|$:

    \begin{center}
        \begin{tabular}{|l|l|l|}
            \hline
            $M=10^2$& $M=10^3$ & $M=10^4$ \\
            \hline
$\widetilde{p}_2(2)=0.37$ & $\widetilde{p}_2(3)=0.319$& $\widetilde{p}_2(4)=0.3102$  \\
            \hline
$e_2\approx 0.065913276015$ & $e_3\approx 0.014913276015$& $e_4\approx 0.006113276015$  \\
            \hline
            \hline
             $M=10^5$ & $M=10^6$&$M=10^7$ \\
            \hline
  $\widetilde{p}_2(5)=0.30416$ &  $\widetilde{p}_2(6)=0.303892$ &$\widetilde{p}_2(7)=0.3041241$ \\
            \hline
  $e_5\approx 0.000073276015$ &  $e_6\approx 0.000194723985$&$e_7\approx 0.000037376015$ \\
            \hline
            \hline
 $M=10^8$& $M=10^9$  & $M=10^{10}$\\
            \hline
 $\widetilde{p}_2(8)=0.30406079$ & 
$\widetilde{p}_2(9)=0.304076699$ & $\widetilde{p}_2(10)=0.304079098$\\
            \hline
 $e_8\approx 0.000025933985$& $e_9\approx  0.000010024985$  &  $e_{10}\approx 0.000007625985$\\
            \hline
            \end{tabular}  
                                     
        \bigskip

{\bf Table 1.} Observed frequency and absolute error of $p_2$ for second order difference equations, using that $p_2=\arctan(\sqrt{2})/\pi
 \approx 0.304086723985$. 
    \end{center}

With the above results, the regression line of $Y=\log(e_m)$ versus $X=\log(M)=m\log(10)$  is 
$Y=-0.505\,X - 1.260$ with $R^2=0.893$. 
The slope is therefore $-0.505\approx -1/2$ as was expected  a priori since, in practice, the absolute error behaves as
$O(M^{-1/2})$, see the Step 2 in the Introduction.

A more detailed explanation of the second step, about the improvement of the Monte Carlo estimations using the least squares method, is as follows: 
the probabilities $p_k$ satisfy some affine relations, like the ones
in Lemma \ref{l:lemmageneral} or the ones in Proposition \ref{p:propoEDiff} below. Then, if we
denote $\mathbf{p}=(p_0,\ldots,p_n)^t\in\R^{n+1}$ it is possible to
write $\mathbf{p}=\mathrm{M}\mathbf{q}+b$ where  $\mathbf{q}\in\R^k$ with
$k\leq n$ is a vector whose components are different elements of
${p_0,\ldots,p_n}$; $\mathrm{M}\in\mathcal{M}_{n\times k}(\R)$; and $b\in \R^k$.
Let
$\widetilde{\mathbf{p}}=(\widetilde{p}_0,\ldots,\widetilde{p}_n)^t$
be the vector with the estimated probabilities obtained by the observed
relative frequencies using the Monte Carlo method. Then, we can find the
least squares solution \cite[Def.~6.1]{Schott} of the the system,
\begin{equation}\label{e:pMqb}
\widetilde{\mathbf{p}}=\mathrm{M}\widehat{\mathbf{q}}+b,
\end{equation}
which is
\begin{equation}\label{e:leastquaresolution}
\widehat{\mathbf{q}}=(\mathrm{M}^t\cdot \mathrm{M})^{-1}\cdot \mathrm{M}^t\cdot (\widetilde{\mathbf{p}}-b),
\end{equation}
see \cite[p.~198]{DB} or \cite[p.~200]{SB}. So we can find some
improved  estimations $\widehat{\mathbf{p}}$, via
\begin{equation}\label{e:improvedestimations}
\widehat{\mathbf{p}}=\mathrm{M}\widehat{\mathbf{q}}+b.
\end{equation}

\noindent Some detailed examples are given in Sections \ref{ss:LRDS},
\ref{ss:LRODEn} and \ref{ss:LRdifference}.

\section{Linear random differential systems}\label{ss:LRDS}

Consider linear differential systems
$
\dot{\mathbf{x}}=A\,\mathbf{x}\mbox{ where }
\mathbf{x}\in\R^n,\,A\in\mathcal{M}_{n\times n}(\R),
$
where $A$ is a random matrix whose entries are   i.i.d.~random variables with  $\mathrm{N}(0,1)$ distribution.
Let~$X$ be the random variable that  counts the number of eigenvalues of $A$ with negative real part, $s(A).$

\begin{propo}\label{p:proposistemes} With the above notations, set $p_k=P(X=k).$ The following holds:
\begin{enumerate}
\item[(a)] $\sum_{k=0}^n p_k=1.$
\item[(b)] For all $k\in\{0,1,\ldots,n\},$ $p_k=p_{n-k}.$
\item[(c)] $\sum_{i\,\mathrm{odd}} p_i=\sum_{i\,\mathrm{even}} p_i=\frac{1}{2}.$
\end{enumerate}
\end{propo}

\begin{proof}
The assertion (a) is trivial. To prove (b) we observe that if a
matrix $A$ has $k$ eigenvalues with negative real part, then $B=-A$
has $n-k$ eigenvalues with negative real part. Calling $q_m$ the
probability that  $B$ has $m$ eigenvalues with negative real part,
we get that $p_m=q_m.$ This is so, because if $X\sim\mathrm{N}(0,1)$
then $-X\sim\mathrm{N}(0,1)$ and as a consequence the entries of $A$
and $B$ have the same distribution. Then, $q_k=p_{n-k}$ and the
result follows.

To see (c) we claim  that $s(A)$ is even if and only if the
determinant of $A$  is positive and, moreover,  $P(\det (A)>0)=1/2.$
From this claim we get the result because
 $\sum_{i\,\text{even}}\,p_i$ is the
probability of $s(A)$ being even. To prove the first part of the
claim notice first that we can assume that
$0\ne\det(A)=\lambda_1\lambda_2\cdots \lambda_n,$ where
$\lambda_1,\lambda_2,\ldots ,\lambda_n$ are the $n$ eigenvalues of
$A.$ We write $\lambda_1\lambda_2\cdots
\lambda_n=(\lambda_1\lambda_2\cdots
\lambda_k)(\lambda_{k+1}\lambda_{k+2}\cdots \lambda_n)$ where
$\lambda_1,\lambda_2,\ldots ,\lambda_k$ are all the real negative
eigenvalues. Observe also that for complex eigenvalues
$\lambda\bar\lambda>0.$ Hence $\lambda_{k+1}\lambda_{k+2}\cdots
\lambda_n>0,$ $\sign(\det(A))=(-1)^k$ and the condition that $s(A)$ is even
is characterized by $\det(A)>0.$ To prove that $P(\det(A)>0)=1/2$
  note that if $B$ is the matrix obtained by changing the
sign of one column of $A$ then $\det(A)\cdot \det(B)<0$ and
hence $P(\det(A)<0)= P(\det(B)>0).$ Furthermore, since
the entries of $A$ and $B$ have the same  distribution we have
$P(\det(B)>0)= P(\det(A)>0)$, thus
$P(\det(A)<0)=P(\det(A)>0)=1/2.$
~\end{proof}

From the above proposition it easily follows:

\begin{corol}\label{c:corolsistemes}
Consider $\dot{\mathbf{x}}=A\,\mathbf{x},$ $\mathbf{x}\in\R^n$ with
$A\in\mathcal{M}_{n\times n}(\R)$ a random matrix with i.i.d.~$\mathrm{N}(0,1)$ entries, let $X$ be the
random variable  defined above and  $p_k=P(X=k)$. Then the probabilities $p_k$
satisfy all the consequences of Lemma \ref{l:lemmageneral}. In particular $E(X)=n/2.$
\end{corol}

\medskip

Now we reproduce some experiments to estimate the probabilities
$p_k$ for low dimensional  cases. We apply the Monte Carlo method,
that is, for each considered dimension $n$, we have generated $10^8$
matrices $A\in\mathcal{M}_{n\times n}(\R)$ whose entries are
pseudo-random numbers simulating the realizations on $n^2$
independent random variables with $\mathrm{N}(0,1)$ distribution.
For each matrix $A$ we have computed the characteristic polynomial,
and counted the number of eigenvalues with negative real part by
using the Routh-Hurwitz zeros counter \cite[p.~1076]{GR}.  We are aware that the stability of the calculation of the coefficients of the characteristic polynomial from the 
entries of a matrix is critical (see [SB, p.378-379] and references therein); 
however we have only used this calculation for low dimensions, namely $n \leq 4$.
For
$n\geq 5$, and in order to decrease the computation time, we have
directly computed numerically the eigenvalues of $A$ and counted the number of
them with negative real part.

For each considered dimension of the phase space $n$, and  in order
to take  advantage of the relations stated in Corollary
\ref{c:corolsistemes}, we can refine the solutions using the least
squares solutions of the inconsistent linear system associated with
these relations when using the observed frequencies obtained by the
Monte Carlo simulation.

We give details of one example. Set $n=7$, for instance. By Corollary
\ref{c:corolsistemes}  we have $p_3=p_4=\frac{1}{2}-p_0-p_1-p_2$;
$p_5=p_2$; $p_6=p_1$ and $p_7=p_0$. So, using the notation
introduced in Section~\ref{ss:method}, we can write
$\mathbf{p}=\mathrm{M}\mathbf{q}+b$, where $\mathbf{p}^t=(p_0,\ldots,p_7)$;
$$
\mathrm{M}=\left(\begin{array}{ccc}
1 & 0 & 0\\
0 & 1 & 0\\
0 & 0 & 1\\
-1& -1 & -1\\
-1& -1 & -1\\
0 & 0 & 1\\
0 & 1 & 0\\
1 & 0 & 0
\end{array}\right); \quad \mathbf{q}=\left(\begin{array}{c}
p_0\\
p_1\\
p_2
\end{array}\right); \,\mbox{ and } b=\left(\begin{array}{c}
0\\
0\\
0\\
\frac{1}{2}\\
\frac{1}{2}\\
0\\
0\\
0
\end{array}\right).
$$
The observed relative frequencies in our Monte Carlo simulation are
{\small
$$\widetilde{\mathbf{p}}^t=\left(\frac{31643}{50000000},\frac{261137}{12500000},
\frac{7124967}{50000000},\frac{1344047}{4000000},\frac{33597117
}{100000000},\frac{14248187 }{100000000},\frac{1043913
}{50000000},\frac{63379 }{100000000}\right).$$} By finding the least
squares solution of the system \eqref{e:pMqb} (\cite[p.~198]{DB} or
\cite[p.~200]{SB}), given by \eqref{e:improvedestimations}, we
obtain {\small $$\widehat{\mathbf{p}}^t= \left(
{\frac{25333}{40000000}},{\frac{
2088461}{100000000}},{\frac{28498121}{200000000}},{\frac{16799573}{
50000000}},{\frac{16799573}{50000000}},{\frac{28498121}{200000000}},{
\frac{2088461}{100000000}},{\frac{25333}{40000000}}
 \right).$$ }
The other cases follow similarly.

\medskip

We summarize the results of our experiments in the Table~2, where
the observed  relative frequencies and the estimates are presented
only up to the fifth decimal  (in the table, and in the whole paper,
frequency stands for relative frequency) because as we already
explained in the introduction, the predicted absolute error will be
of order $10^{-4}$. Observe that in the cases $n=1,2$ the true probabilities are known. We include the results of the Monte Carlo simulations for completeness, but  it makes no sense to apply the least squares method.

\begin{center}
\begin{tabular}{|l|l|l|l|}
\hline
Dimension & Observed frequency &  Least squares & Relations (Corol. \ref{c:corolsistemes})\\
\hline
\hline
$n=1$ & $\widetilde{p}_0=0.49996$ & &$p_0=0.5$\\
{}  & $\widetilde{p}_1=0.50004$ &   &$p_1=0.5$\\
\hline
$n=2$ & $\widetilde{p}_0=0.24999$ & &$p_0=0.25$\\
{}  & $\widetilde{p}_1=0.50006$ &   &$p_1=0.5$\\
{}  & $\widetilde{p}_2=0.24995$ &   &$p_2=0.25$\\
\hline
$n=3$ & $\widetilde{p}_0=0.10447$ &  $\widehat{p}_0=0.10450$ &$p_0$\\
{}  & $\widetilde{p}_1=0.39542$ &  $\widehat{p}_1=0.39550$ &$p_1=\frac{1}{2}-p_0$\\
{}  & $\widetilde{p}_2=0.39557$ &  $\widehat{p}_2=0.39550$ &$p_2=\frac{1}{2}-p_0$\\
{}  & $\widetilde{p}_3=0.10454$ &  $\widehat{p}_3=0.10450$ &$p_3=p_0$\\
\hline
$n=4$ & $\widetilde{p}_0=0.03722$ &  $\widehat{p}_0=0.03721$ &$p_0$\\
{}  & $\widetilde{p}_1=0.25009$ &  $\widehat{p}_1=0.25000$ &$p_1=\frac{1}{4}$\\
{}  & $\widetilde{p}_2=0.42556$ &  $\widehat{p}_2=0.42558$ &$p_2=\frac{1}{2}-2p_0$\\
{}  & $\widetilde{p}_3=0.24998$ &  $\widehat{p}_3=0.25000$ &$p_3=\frac{1}{4}$\\
{}  & $\widetilde{p}_4=0.03715$ &  $\widehat{p}_4=0.03721$ &$p_4=p_0$\\
\hline
$n=5$ & $\widetilde{p}_0=0.01126$ & $\widehat{p}_0=0.01126$ &  $p_0$\\
{}  & $\widetilde{p}_1=0.13028$ & $\widehat{p}_1=0.13024$ &  $p_1$\\
{}  & $\widetilde{p}_2=0.35848$ & $\widehat{p}_2=0.35850$ &  $p_2=\frac{1}{2}-p_0-p_1$\\
{}  & $\widetilde{p}_3=0.35852$ & $\widehat{p}_3=0.35850$ &  $p_3=\frac{1}{2}-p_0-p_1$\\
{}  & $\widetilde{p}_4=0.13020$ & $\widehat{p}_4=0.13024$ &  $p_4=p_1$\\
{}  & $\widetilde{p}_5=0.01126$ & $\widehat{p}_5=0.01126$ &  $p_5=p_0$\\
\hline
\end{tabular}
\end{center}

\begin{center}
\begin{tabular}{|l|l|l|l|}
\hline
Dimension & Observed frequency &  Least squares & Relations (Corol. \ref{c:corolsistemes})\\
\hline
\hline
$n=6$ & $\widetilde{p}_0=0.00289$ & $\widehat{p}_0=0.00288$ &  $p_0$\\
{}  & $\widetilde{p}_1=0.05675$ & $\widehat{p}_1=0.05678$ &  $p_1$\\
{}  & $\widetilde{p}_2=0.24710$ & $\widehat{p}_2=0.24712$ &  $p_2=\frac{1}{4}-p_0$\\
{}  & $\widetilde{p}_3=0.38642$ & $\widehat{p}_3=0.38644$ &  $p_3=\frac{1}{2}-2p_1$\\
{}  & $\widetilde{p}_4=0.24714$ & $\widehat{p}_4=0.24712$ &  $p_4=\frac{1}{4}-p_0$\\
{}  & $\widetilde{p}_5=0.56810$ & $\widehat{p}_5=0.05678$ &  $p_5=p_1$\\
{}  & $\widetilde{p}_6=0.00289$ & $\widehat{p}_6=0.00288$ &  $p_6=p_0$\\
\hline
$n=7$ & $\widetilde{p}_0=0.00063$ & $\widehat{p}_0=0.00063$ &  $p_0$\\
{}  & $\widetilde{p}_1=0.02089$ & $\widehat{p}_1=0.02088$ &  $p_1$\\
{}  & $\widetilde{p}_2=0.14250$ & $\widehat{p}_2=0.14249$ &  $p_2$\\
{}  & $\widetilde{p}_3=0.33601$ & $\widehat{p}_3=0.33600$ &  $p_3=\frac{1}{2}-p_0-p_1-p_2$\\
{}  & $\widetilde{p}_4=0.33597$ & $\widehat{p}_4=0.33600$ &  $p_4=\frac{1}{2}-p_0-p_1-p_2$\\
{}  & $\widetilde{p}_5=0.14248$ & $\widehat{p}_5=0.14249$ &  $p_5=p_2$\\
{}  & $\widetilde{p}_6=0.02088$ & $\widehat{p}_6=0.02088$ &  $p_6=p_1$\\
{}  & $\widetilde{p}_7=0.00063$ & $\widehat{p}_7=0.00063$ &  $p_7=p_0$\\
\hline
$n=8$ & $\widetilde{p}_0=0.00012$ & $\widehat{p}_0=0.00012$ &  $p_0$\\
{}  & $\widetilde{p}_1=0.00651$ & $\widehat{p}_1=0.00650$ &  $p_1$\\
{}  & $\widetilde{p}_2=0.06948$ & $\widehat{p}_2=0.06948$ &  $p_2$\\
{}  & $\widetilde{p}_3=0.24356$ & $\widehat{p}_3=0.24350$ &  $p_3=\frac{1}{4}-p_1$\\
{}  & $\widetilde{p}_4=0.36080$ & $\widehat{p}_4=0.36080$ &  $p_4=\frac{1}{2}-2p_0-2p_2$\\
{}  & $\widetilde{p}_5=0.24346$ & $\widehat{p}_5=0.24350$ &  $p_5=\frac{1}{4}-p_1$\\
{}  & $\widetilde{p}_6=0.06946$ & $\widehat{p}_6=0.06948$ &  $p_6=p_2$\\
{}  & $\widetilde{p}_7=0.00650$ & $\widehat{p}_7=0.00650$ &  $p_7=p_1$\\
{}  & $\widetilde{p}_8=0.00012$ & $\widehat{p}_8=0.00012$ &  $p_8=p_0$\\
\hline
$n=9$ & $\widetilde{p}_0=0.00002$ & $\widehat{p}_0=0.00002$ &  $p_0$\\
{}  & $\widetilde{p}_1=0.00171$ & $\widehat{p}_1=0.00171$ &  $p_1$\\
{}  & $\widetilde{p}_2=0.02880$ & $\widehat{p}_2=0.02879$ &  $p_2$\\
{}  & $\widetilde{p}_3=0.14952$ & $\widehat{p}_3=0.14955$ &  $p_3$\\
{}  & $\widetilde{p}_4=0.31987$ & $\widehat{p}_4=0.31993$ &  $p_4=\frac{1}{2}-p_0-p_1-p_2-p_3$\\
{}  & $\widetilde{p}_5=0.31999$ & $\widehat{p}_5=0.31993$ &  $p_5=\frac{1}{2}-p_0-p_1-p_2-p_3$\\
{}  & $\widetilde{p}_6=0.14958$ & $\widehat{p}_6=0.14955$ &  $p_6=p_3$\\
{}  & $\widetilde{p}_7=0.02878$ & $\widehat{p}_7=0.02879$ &  $p_7=p_2$\\
{}  & $\widetilde{p}_8=0.00171$ & $\widehat{p}_8=0.00171$ &  $p_8=p_1$\\
{}  & $\widetilde{p}_9=0.00002$ & $\widehat{p}_9=0.00002$ &  $p_9=p_0$\\
\hline
\end{tabular}
\end{center}

\begin{center}
\begin{tabular}{|l|l|l|l|}
\hline
Dimension & Observed frequency &  Least squares & Relations (Corol. \ref{c:corolsistemes})\\
\hline
\hline
$n=10$ & $\widetilde{p}_0=0$ & $\widehat{p}_0=0$ &  $p_0$\\
{}  & $\widetilde{p}_1=0.00038$ & $\widehat{p}_1=0.00038$ &  $p_1$\\
{}  & $\widetilde{p}_2=0.01015$ & $\widehat{p}_2=0.01015$ &  $p_2$\\
{}  & $\widetilde{p}_3=0.07850$ & $\widehat{p}_3=0.07850$ &  $p_3$\\
{}  & $\widetilde{p}_4=0.23987$ & $\widehat{p}_4=0.23985$ &  $p_4=\frac{1}{4}-p_0-p_2$\\
{}  & $\widetilde{p}_5=0.34224$ & $\widehat{p}_5=0.34224$ &  $p_5=\frac{1}{2}-2p_1-2p_3$\\
{}  & $\widetilde{p}_6=0.23984$ & $\widehat{p}_6=0.23985$ &  $p_6=\frac{1}{4}-p_0-p_2$\\
{}  & $\widetilde{p}_7=0.07849$ & $\widehat{p}_7=0.07850$ &  $p_7=p_3$\\
{}  & $\widetilde{p}_8=0.01015$ & $\widehat{p}_8=0.01015$ &  $p_8=p_2$\\
{}  & $\widetilde{p}_9=0.00038$ & $\widehat{p}_9=0.00038$ &  $p_9=p_1$\\
{}  & $\widetilde{p}_{10}=0$ & $\widehat{p}_{10}=0$ &  $p_{10}=p_0$\\
\hline
\end{tabular}


\bigskip

{\bf Table 2.} Linear stability indexes for linear random differential systems.
\end{center}

\section{Linear random differential equations of order~$n$}\label{ss:LRODEn}
In this section we consider linear random homogeneous differential equations of order $n$
\begin{equation}\label{eqdiflineal}
A_n x^{(n)}+ A_{n-1} x^{(n-1)}+ \cdots + A_2 x''+A_1 x'+A_0
x=0,\end{equation} where $x=x(t)$, the derivatives are taken in
respect to $t$, and $A_j$ are again
 i.i.d.~random variables with $\mathrm{N}(0,1)$ distribution.

To get the stability index for these differential equations, we need we only need to know the probability distributions of the number of
roots with negative real part  of its associated random characteristic polynomial:
$$
Q(\lambda)=A_n\lambda^{n}+A_{n-1}\lambda^{n-1}+\cdots
+A_1\lambda+A_0.$$ Let  $X$ be the random variable that
counts the number of roots of $Q(\lambda)$ with negative real parts
and define $p_k=P(X=k)$ for $k=0,1,\ldots,n.$

\begin{propo}\label{p:propoEDOL}
Set $p_k=P(X=k),$ where $X$ is the  random variable defined above.
Then
\begin{enumerate}
\item[(a)] $\sum_{k=0}^n p_k=1.$
\item[(b)] For all $k\in\{0,1,\ldots,n\},$ $p_k=p_{n-k}.$
\item[(c)] $\sum_{i\,\mathrm{odd}} p_i=\sum_{i\,\mathrm{even}} p_i=\frac{1}{2}.$
\end{enumerate}
\end{propo}

\begin{proof} The proof of (a) is trivial. To prove (b) consider
equation (\ref{eqdiflineal}) with its characteristic
polynomial $Q(\lambda)$ and also the new differential equation
\begin{equation}\label{eqdiflinealprima}
(-1)^nA_n x^{(n)}+(-1)^{n-1} A_{n-1} x^{(n-1)}+ \cdots + A_2 x''-A_1
x'+A_0 x=0\end{equation}  with its characteristic polynomial
$Q^*(\lambda)=Q(-\lambda)=(-1)^nA_n\lambda^{n}+
(-1)^{n-1}A_{n-1}\lambda^{n-1}+\cdots -A_1\lambda+A_0.$ Since
$Q(\lambda)=0$ if and only if $Q^*(-\lambda)=0$ we get that
$p_k=p^*_{n-k}$ where $p^*_i$ the probability that $Q^*(\lambda)$
has $i$ roots with negative real part. But also $p_k=p^*_k$ because
the coefficients of the equations (\ref{eqdiflineal}) and
(\ref{eqdiflinealprima}) have the same distributions. Hence, the
result follows.

Similarly, as in the proof of $(c)$ of Proposition \ref{p:proposistemes}, we
observe that the polynomial $Q(\lambda)$ has an odd number of roots
with negative real part if and only if $A_0\cdot A_n<0,$ because we can neglect the case of having
some roots with zero real part.
Since the
coefficients of~(\ref{eqdiflineal}) are symmetric independent random variables, the probability
that $Q(\lambda)$ has an odd number of roots with negative real
part is
$$
P(\{A_0>0\}\cap\, \{A_n<0\})+P(\{A_0<0\}\cap\, \{A_n>0\})=\frac12\times\frac12+\frac12\times\frac12=\frac12.
$$
\end{proof}

\begin{corol}\label{c:propoEDOL3}
    Consider the linear random homogeneous differential equation of
    order $n$ \eqref{eqdiflineal}, with all $A_i$ being
    i.i.d.~$\mathrm{N}(0,1)$ random variables,   let $X$ be
    defined above, and set $p_k=P(X=k).$ Then the probabilities $p_k$
    satisfy all the conclusions of Lemma \ref{l:lemmageneral}. In particular $E(X)=n/2.$
\end{corol}

For each $n,$ let $r_n$ be  the probability of the origin to be a
global stable attractor (asymptotically stable equilibrium) for
\eqref{eqdiflineal}, that is $r_n=p_n.$ By Proposition
\ref{p:propoEDOL}(b) this probability coincides with the probability of
being a repeller because $p_n=p_0.$ Our results in
Proposition~\ref{p:nova}  seem to indicate that $r_n$  decreases
with $n.$ Before proving this proposition we need a preliminary
result.

\begin{lem}\label{l:lemmaauxq}
    Let $U,V,S$ and $T$ be i.i.d.~random variables with standard  normal  distribution. Then
    $p^+:=P(U>0;V>0;S>0;T>0; UT-SV>0)=1/32$.
\end{lem}
\begin{proof} Set $\mathcal{A}^\pm=\{U>0;V>0;S>0;T>0; \pm(UT-SV)>0\},$    and  $\mathcal{A}^0=\{U>0;V>0;S>0;T>0; UT-SV=0\}.$
    Denote by $p^{\pm}=P(\mathcal{A}^\pm)$ and $p^0=P(\mathcal{A}^0).$
    Then, since  $p^0=0$ and
     $ \mathcal{A^-}\cup\mathcal{A}^0\cup\mathcal{A}^+=\{U>0;V>0;S>0;T>0\}$ it
      holds that $p^++p^-=(1/2)^4=1/16.$ To end the proof it suffices to show that $p^+=p^-.$

    Notice first that
    \begin{align*}
    \mathcal{A}^+&=\{U>0;V>0;S>0;T>0; UT-SV>0\}=\{V>0;S>0;T>0; UT-SV>0\},\\
        \mathcal{A}^-&=\{U>0;V>0;S>0;T>0; UT-SV<0\}=\{U>0;S>0;T>0; SV-UT>0\}.
    \end{align*}
This is so, because for instance in the definition of
$\mathcal{A}^+$, the last inequality can also be written as
$U>SV/T>0$ and from it we know that the condition $U>0$ can be
removed. Finally, interchanging $U$ and $V$ and $S$ and $T$ we get
the same relations in the definitions of $\mathcal{A}^+$ and
$\mathcal{A}^-$. Since all variables are independent $\mathrm{N}(0,1)$, both
sets have the same probability and $p^+=p^-,$ as we wanted to prove.
    ~\end{proof}

\begin{propo}\label{p:nova} With the above notations,
  $r_n\leq {1}/{2^n},$ for all $n\geq 1.$  Moreover, $r_1=1/2,$ $r_2=1/4,$ $r_3=1/16$ and
  $r_4<r_3/2=1/32.$
\end{propo}

\begin{proof}[Proof of Proposition~\ref{p:nova}]
Notice  that $r_n$ is the probability that the characteristic polynomial  $Q(\lambda)$,
 associated with the random differential equation \eqref{eqdiflineal}, is a Hurwitz stable polynomial;
 that is
$ r_n=P(\mbox{Every root of }Q(\lambda)\mbox{ belongs to }
\mathfrak{R}^-), $ where $\mathfrak{R}^-=\{z\in\mathbb{C}\mbox{ such
that } \mathrm{Re}(z)<0\}$. It is well--known that a necessary
condition for a polynomial to have every root in $\mathfrak{R}^-$ is
that all its coefficients have the same sign. This is so because it
holds for polynomials of degree 1 and 2, and this property is
preserved when we multiply two polynomials satisfying it. Hence,
\begin{multline}\label{eq:sets}
\left\{\mbox{$A_0,\ldots ,A_n$ such that all roots of
}P(\lambda)\mbox{ are in } \mathfrak{R}^-\right\}
\subset\\\left\{\bigcap\limits_{i=0}^n \{A_i<0\}\right\}\bigcup
\left\{\bigcap\limits_{i=0}^n \{A_i>0\}\right\}.
\end{multline}
Since the variables $A_i$ are independent and symmetric
$$
P\left(\bigcap\limits_{i=0}^n
\{A_i<0\}\right)=P\left(\bigcap\limits_{i=0}^n \{A_i>0\}\right)=
\dfrac{1}{2^{n+1}}.
$$
As a consequence,
$$
r_n\leq P\left(\bigcap\limits_{i=0}^n
\{A_i<0\}\right)+P\left(\bigcap\limits_{i=0}^n \{A_i>0\}\right)=
\dfrac{1}{2^{n}},
$$
and the first statement follows.

The equalities $r_1=1/2$ and  $r_2=1/4$ are a simple consequence
that for $n=1,2$ the inclusion~\eqref{eq:sets} is an equality.

Let us prove  that $r_3=p_3=1/16.$  By using the Routh-Hurwitz criterion
    \cite[p.~1076]{GR}, it can be seen that  $a_3\lambda^3+a_2\lambda^2+a_1\lambda+a_0 $
    has every root in $\mathfrak{R}^-$  if and only if all its coefficients
    have the same sign and moreover $a_1a_2-a_0a_3>0$. Hence, $p_3=p_3^-+p_3^+,$
    where $p_3^-:=P(A_0<0;A_1<0;A_2<0;A_3<0; A_1A_2-A_0A_3>0);$ and
    $p_3^+:=P(A_0>0;A_1>0;A_2>0;A_3>0; A_1A_2-A_0A_3>0),$  with all the $A_i$ being $\mathrm{N}(0,1)$ distributed and independent. Due to their
    symmetry, the random variables $A_i$ and $-A_i$, for $i=0,\ldots,3$ have
    the same distribution and hence $p_3^+=p_3^-.$ Therefore $p_3=2p_3^+$.
    The result follows now by Lemma \ref{l:lemmaauxq}, which gives
    $p_3^+=1/32$.

Let us prove that $r_4<r_3/2.$ To compare both probabilities, here
it will be more convenient to write the coefficients of the
polynomials with subscripts with increasing ordering, that is
$q_n(x)=a_0x^n+a_1x^{n-1}+\cdots+a_{n-1}x+a_n.$ With this notation, which also respects the traditional notation when writing the Hurwitz matrices,
and when $a_0>0,$ the Routh-Hurwitz conditions to have stability index $n$  for $n=3,4$ are
precisely that the principal minors of the following matrices
\[
\left(
  \begin{array}{ccc}
    a_1 & a_3 & 0  \\
    a_0 & a_2 &  0 \\
    0 & a_1 & a_3
  \end{array}
\right)\qquad\mbox{and}\qquad \left(
  \begin{array}{cccc}
    a_1 & a_3 & 0 & 0 \\
    a_0 & a_2 & a_4 & 0 \\
    0 & a_1 & a_3 & 0 \\
    0 & a_0 & a_2 & a_4 \\
  \end{array}
\right),
\]
are positive, where the left-hand one corresponds to the case $n=3$ and the
other to the case when $n=4.$ Hence, these conditions when $a_0>0$ and for
$n=3$ are: $a_1>0, a_1a_2-a_0a_3>0$ and $a_3>0.$ Similarly, for
$n=4$ the conditions are $a_1>0, a_1a_2-a_0a_3>0,$ $a_3(a_1a_2-a_0a_3)-a_4a_1^2>0$
and $a_4>0.$

Consider now, for $n=3,4$, the random polynomials
$Q_n(x)=\tilde{A}_0x^n+\tilde{A}_1x^{n-1}+\cdots+\tilde{A}_{n-1}x+\tilde{A}_n,$ where  $\tilde{A}_i\sim
\mathrm{N}(0,1)$ and are independent (notice that with this notation each coefficient $\tilde{A}_i$ is the coefficient $A_{n-i}$ of the characteristic polynomial). For simplicity we denote with the same name
the coefficients of $Q_3$ and $Q_4$ although they are different
random variables. As above, $r_3=2p_3^+,$ and $r_4=2p_4^+,$  where
$p_k^+=P(\mathcal{A}^+_k),$ with
\begin{align*}
\mathcal{A}^+_3&=\{\tilde{A}_0>0;\tilde{A}_1>0;\tilde{A}_3>0; \tilde{A}_1\tilde{A}_2-\tilde{A}_0\tilde{A}_3>0\},\\
\mathcal{A}^+_4&=\{\tilde{A}_0>0;\tilde{A}_1>0;\tilde{A}_3>\tilde{A}_4\tilde{A}_1^2/(\tilde{A}_1\tilde{A}_2-\tilde{A}_0\tilde{A}_3);
\tilde{A}_1\tilde{A}_2-\tilde{A}_0\tilde{A}_3>0,\tilde{A}_4>0\}.
\end{align*}
Notice that if we define
\[
\mathcal{B}=\{\tilde{A}_0>0;\tilde{A}_1>0;\tilde{A}_3>0; \tilde{A}_1\tilde{A}_2-\tilde{A}_0\tilde{A}_3>0; \tilde{A}_4>0\}
\]
it is clear that $P(\mathcal{B})=p_3^+/2$ and, moreover
$\mathcal{A}^+_4\subset \mathcal{B}_3,$ with the inclusion being strict.
Since the joint density is positive and $\mathcal{B}\cap (\mathcal{A}^+_4)^c$ has positive Lebesgue measure, we have $P(\mathcal{B}\cap (\mathcal{A}^+_4)^c)>0.$ Thus $P(\mathcal{A}^+_4)<P(\mathcal{B})$, and hence
$p_4^+=P(\mathcal{A}^+_4)<P(\mathcal{B})=p_3^+/2,$ and
$r_4<r_3/2,$ as we wanted to show.~\end{proof}

\begin{corol}\label{p:propoEDOL3}
Consider a linear random homogeneous differential equation of order
$n=3$ and the random variable $X$  defined above. Then
$p_0=p_3={1}/{16}$ and $p_1=p_2={7}/{16}$.
    \end{corol}
\begin{proof}
    By the above proposition, for $n=3,$  $p_0=p_3=r_3=1/16.$ Hence, by Proposition~\ref{p:propoEDOL}, $p_1=p_2=7/16.$
\end{proof}

\medskip

The  computations in this case are similar to the ones of the
previous section and the obtained results are summarized in Table 3.
 We only give some comments for the cases $n=8$ and $10,$ where we
have encountered that the vectors~$\widehat{\mathbf{p}}$  have
negative and very small entries. This has occurred because the
observed frequencies obtained by the Monte Carlo approach
corresponding  to these probabilities are not enough accurate. For
this reason,  we have made a new optimization step.  As before, we use the least squares method to obtain a vector $\widehat{\mathbf{p}}$. However, if negative entries appear in this vector (which is clearly objectionable), we force them to be zero and find a new least squares estimate, which still respects the original linear constraints.

\smallskip

We explain this process for the $n=8$ order case: The observed
relative frequencies vector obtained by the Monte Carlo method is
{\small \begin{multline*}\widetilde{\mathbf{p}}^t=
\Bigg({\frac{1}{50000000}}, \frac{6599}{50000000},
\frac{1159359}{50000000}, \frac{4996163}{20000000},
\frac{45377377}{100000000},\\ \frac{4995607}{ 20000000},
\frac{2318357}{100000000}, \frac{13497}{100000000},
\frac{1}{100000000} \Bigg).
\end{multline*}}
The relations stated in Corollary \ref{c:propoEDOL3}  are
$p_3=1/4-p_1$, $p_4=1/2-2p_0-2p_2$, $p_5=p_3$, $p_6=p_2$, $p_7=p_1$,
$p_8=p_0$. By solving the system \eqref{e:pMqb} with

$$
\mathrm{M}=\left(\begin{array}{ccc}
1 & 0 & 0\\
0 & 1 & 0\\
0 & 0 & 1\\
0& -1 & 0\\
-2& 0 & -2\\
0& -1 & 0\\
0 & 0 & 1\\
0 & 1 & 0\\
1 & 0 & 0
\end{array}\right); \quad \mathbf{q}=\left(\begin{array}{c}
p_0\\
p_1\\
p_2
\end{array}\right); \,\mbox{ and } b=\left(\begin{array}{c}
0\\
0\\
0\\
\frac{1}{2}\\
\frac{1}{4}\\
\frac{1}{2}\\
0\\
0\\
0
\end{array}\right)
$$
we obtain
$$
\widehat{\mathbf{q}}^t=\left(
-{\frac{5779}{200000000}},\,{\frac{13569}{
80000000}},\,{\frac{4631293}{200000000}}\right).
$$
Hence, by \eqref{e:improvedestimations} we get
{\small\begin{multline*} \widehat{\mathbf{p}}= \Bigg(
\frac{-5779}{200000000},\frac{
13569}{80000000},\frac{4631293}{200000000},\frac{19986431}{
80000000},\frac{22687243}{50000000},\\\frac{19986431}{80000000},
\frac{4631293}{200000000},\frac{13569}{80000000},\frac{-5779}{
200000000}\Bigg).
\end{multline*}}
\noindent So we impose that $p_0=p_8=0$. Thus we have
$p_3=p_5=1/4-p_1$, $p_4=1/2-2p_0-2p_2=1/2-2p_2$, $p_6=p_2$ and
$p_7=p_1$. We find the least squares solution of the system
$$
\left(\begin{array}{c}
\widehat{p}_1\\
\widehat{p}_2\\
\widehat{p}_3\\
\widehat{p}_4\\
\widehat{p}_5\\
\widehat{p}_6\\
\widehat{p}_7
\end{array}\right)= \left( \begin {array}{cc}  1&0
\\ 0&1\\ -1&0\\ 0
&-2\\ -1&0\\ 0&1
\\ 1&0\\ \end {array} \right)\cdot \left(\begin{array}{c}
\widehat{p}^*_1\\
\widehat{p}^*_2
\end{array}\right)+\left(\begin{array}{c}
0\\
0\\
\frac{1}{4}\\
\frac{1}{2}\\
\frac{1}{4}\\
0\\
0
\end{array}\right).
$$
Using \eqref{e:leastquaresolution} and \eqref{e:improvedestimations}
we obtain {\small $$ \widehat{\mathbf{p}}^*= \left(
0,\,{\frac{13569}{80000000}},\,{\frac{
13882321}{600000000}},\,{\frac{19986431}{80000000}},\,{\frac{136117679}{
300000000}},\,{\frac{19986431}{80000000}},\,{\frac{13882321}{600000000}},\,{
\frac{13569}{80000000}},\,0\right)
 $$}
$$
\simeq \left( 0,\, 0.00017,\, 0.02314 ,\, 0.24983,\,
0.45373,\,0.24983,\, 0.02314,\,
 0.00017,\, 0\right).
$$

The $n=10$ case follows analogously.

\begin{center}
\begin{tabular}{|l|l|l|l|}
\hline
Dimension & Observed frequency &  Least squares & Relations (Corol. \ref{c:propoEDOL3} and \ref{p:propoEDOL3})\\
\hline
\hline
$n=1$ & $\widetilde{p}_0=0.49997$ & &$p_0=0.5$\\
{}  & $\widetilde{p}_1=0.50003$ &   &$p_1=0.5$\\
\hline
$n=2$ & $\widetilde{p}_0=0.24994$ & &$p_0=0.25$\\
{}  & $\widetilde{p}_1=0.49999$ &   &$p_1=0.5$\\
{}  & $\widetilde{p}_2=0.25007$ &   &$p_2=0.25$\\
\hline
$n=3$ & $\widetilde{p}_0=0.06252$ & &$p_0=\frac1{16}=0.0625$\\
{}  & $\widetilde{p}_1=0.43743$ &   &$p_1=\frac7{16}=0.4375$\\
{}  & $\widetilde{p}_2=0.43756$ &   &$p_2=\frac7{16}=0.4375$\\
{}  & $\widetilde{p}_3=0.06249$ & &$p_3=\frac1{16}=0.0625$\\
\hline
$n=4$ & $\widetilde{p}_0=0.00928$ &  $\widehat{p}_0=0.00925$ &$p_0$\\
{}  & $\widetilde{p}_1=0.24998$ &  $\widehat{p}_1=0.25$ &$p_1=\frac{1}{4}$\\
{}  & $\widetilde{p}_2=0.48152$ &  $\widehat{p}_2=0.48150$ &$p_2=\frac{1}{2}-2p_0$\\
{}  & $\widetilde{p}_3=0.24994$ &  $\widehat{p}_3=0.25$ &$p_3=\frac{1}{4}$\\
{}  & $\widetilde{p}_4=0.00929$ &  $\widehat{p}_4=0.00925$ &$p_4=p_0$\\
\hline
$n=5$ & $\widetilde{p}_0=0.00071$ & $\widehat{p}_0=0.00071$ &  $p_0$\\
{}  & $\widetilde{p}_1=0.08405$ & $\widehat{p}_1=0.08404$ &  $p_1$\\
{}  & $\widetilde{p}_2=0.41526$ & $\widehat{p}_2=0.41525$ &  $p_2=\frac{1}{2}-p_0-p_1$\\
{}  & $\widetilde{p}_3=0.41523$ & $\widehat{p}_3=0.41525$ &  $p_3=\frac{1}{2}-p_0-p_1$\\
{}  & $\widetilde{p}_4=0.08404$ & $\widehat{p}_4=0.08404$ &  $p_4=p_1$\\
{}  & $\widetilde{p}_5=0.00071$ & $\widehat{p}_5=0.00071$ &  $p_5=p_0$\\
\hline
\end{tabular}
\end{center}

\begin{center}
\begin{tabular}{|l|l|l|l|}
\hline
Dimension & Observed frequency &  Least squares & Relations (Corol. \ref{c:propoEDOL3} and \ref{p:propoEDOL3})\\
\hline
\hline
$n=6$ & $\widetilde{p}_0=0.00003$ & $\widehat{p}_0=0.00005$ &  $p_0$\\
{}  & $\widetilde{p}_1=0.01723$ & $\widehat{p}_1=0.01720$ &  $p_1$\\
{}  & $\widetilde{p}_2=0.24994$ & $\widehat{p}_2=0.24995$ &  $p_2=\frac{1}{4}-p_0$\\
{}  & $\widetilde{p}_3=0.46562$ & $\widehat{p}_3=0.46560$ &  $p_3=\frac{1}{2}-2p_1$\\
{}  & $\widetilde{p}_4=0.24993$ & $\widehat{p}_4=0.24995$ &  $p_4=\frac{1}{4}-p_0$\\
{}  & $\widetilde{p}_5=0.01723$ & $\widehat{p}_5=0.01720$ &  $p_5=p_1$\\
{}  & $\widetilde{p}_6=0.00003$ & $\widehat{p}_6=0.00005$ &  $p_6=p_0$\\
\hline
$n=7$ & $\widetilde{p}_0=0$ & $\widehat{p}_0=0$ &  $p_0$\\
{}  & $\widetilde{p}_1=0.00200$ & $\widehat{p}_1=0.00200$ &  $p_1$\\
{}  & $\widetilde{p}_2=0.09571$ & $\widehat{p}_2=0.09572$ &  $p_2$\\
{}  & $\widetilde{p}_3=0.40224$ & $\widehat{p}_3=0.40228$ &  $p_3=\frac{1}{2}-p_0-p_1-p_2$\\
{}  & $\widetilde{p}_4=0.40233$ & $\widehat{p}_4=0.40228$ &  $p_4=\frac{1}{2}-p_0-p_1-p_2$\\
{}  & $\widetilde{p}_5=0.09573$ & $\widehat{p}_5=0.09572$ &  $p_5=p_2$\\
{}  & $\widetilde{p}_6=0.00199$ & $\widehat{p}_6=0.00200$ &  $p_6=p_1$\\
{}  & $\widetilde{p}_7=0$ & $\widehat{p}_7=0$ &  $p_7=p_0$\\
\hline
$n=8$ & $\widetilde{p}_0=0$ & $\widehat{p}_0^*=0$ &  $p_0$\\
{}  & $\widetilde{p}_1=0.00013$ & $\widehat{p}_1^*=0.00017$ &  $p_1$\\
{}  & $\widetilde{p}_2=0.02319$ & $\widehat{p}_2^*=0.02314$ &  $p_2$\\
{}  & $\widetilde{p}_3=0.24981$ & $\widehat{p}_3^*=0.24983$ &  $p_3=\frac{1}{4}-p_1$\\
{}  & $\widetilde{p}_4=0.45377$ & $\widehat{p}_4^*=0.45372$ &  $p_4=\frac{1}{2}-2p_0-2p_2$\\
{}  & $\widetilde{p}_5=0.24978$ & $\widehat{p}_5^*=0.24983$ &  $p_5=\frac{1}{4}-p_1$\\
{}  & $\widetilde{p}_6=0.02318$ & $\widehat{p}_6^*=0.02314$ &  $p_6=p_2$\\
{}  & $\widetilde{p}_7=0.00013$ & $\widehat{p}_7^*=0.00017$ &  $p_7=p_1$\\
{}  & $\widetilde{p}_8=0$ & $\widehat{p}_8^*=0$ &  $p_8=p_0$\\
\hline
$n=9$ & $\widetilde{p}_0=0$ & $\widehat{p}_0=0$ &  $p_0$\\
{}  & $\widetilde{p}_1=0.00001$ & $\widehat{p}_1=0.00005$ &  $p_1$\\
{}  & $\widetilde{p}_2=0.00336$ & $\widehat{p}_2=0.00337$ &  $p_2$\\
{}  & $\widetilde{p}_3=0.10337$ & $\widehat{p}_3=0.10335$ &  $p_3$\\
{}  & $\widetilde{p}_4=0.39328$ & $\widehat{p}_4=0.39328$ &  $p_4=\frac{1}{2}-p_0-p_1-p_2-p_3$\\
{}  & $\widetilde{p}_5=0.39328$ & $\widehat{p}_5=0.39328$ &  $p_5=\frac{1}{2}-p_0-p_1-p_2-p_3$\\
{}  & $\widetilde{p}_6=0.10332$ & $\widehat{p}_6=0.10335$ &  $p_6=p_3$\\
{}  & $\widetilde{p}_7=0.00338$ & $\widehat{p}_7=0.00337$ &  $p_7=p_2$\\
{}  & $\widetilde{p}_8=0$ & $\widehat{p}_8=0.00005$ &  $p_8=p_1$\\
{}  & $\widetilde{p}_9=0$ & $\widehat{p}_9=0$ &  $p_9=p_0$\\
\hline
\end{tabular}
\end{center}

\begin{center}
\begin{tabular}{|l|l|l|l|}
\hline
Dimension & Observed frequency &  Least squares & Relations (Corol. \ref{c:propoEDOL3} and \ref{p:propoEDOL3})\\
\hline
\hline

$n=10$ & $\widetilde{p}_0=0$ & $\widehat{p}_0^*=0$ &  $p_0$\\
{}  & $\widetilde{p}_1=0$ & $\widehat{p}_1^*=0.00002$ &  $p_1$\\
{}  & $\widetilde{p}_2=0.00030$ & $\widehat{p}_2^*=0.00028$ &  $p_2$\\
{}  & $\widetilde{p}_3=0.02784$ & $\widehat{p}_3^*=0.02787$ &  $p_3$\\
{}  & $\widetilde{p}_4=0.24976$ & $\widehat{p}_4^*=0.24972$ &  $p_4=\frac{1}{4}-p_0-p_2$\\
{}  & $\widetilde{p}_5=0.44421$ & $\widehat{p}_5^*=0.44422$ &  $p_5=\frac{1}{2}-2p_1-2p_3$\\
{}  & $\widetilde{p}_6=0.24973$ & $\widehat{p}_6^*=0.24972$ &  $p_6=\frac{1}{4}-p_0-p_2$\\
{}  & $\widetilde{p}_7=0.02787$ & $\widehat{p}_7^*=0.02787$ &  $p_7=p_3$\\
{}  & $\widetilde{p}_8=0.00029$ & $\widehat{p}_8^*=0.00028$ &  $p_8=p_2$\\
{}  & $\widetilde{p}_9=0$ & $\widehat{p}_9^*=0.00002$ &  $p_9=p_1$\\
{}  & $\widetilde{p}_{10}=0$ & $\widehat{p}_{10}^*=0$ &  $p_{10}=p_0$\\
\hline
\end{tabular}

\bigskip

{\bf Table 3.} Stability indexes for order $n$ linear random homogeneous differential equations.
\end{center}

\section{Linear random maps}\label{ss:LRMaps}

In order to keep the approach   of the preceding sections,  we suggest to consider random linear discrete dynamical systems
of the form
\begin{equation}\label{sistemadiscret}\mathcal{B}\,\mathbf{x}_{k+1}=A\,\mathbf{x}_k\,\,
\text{where}\,\, \mathbf{x}\in\R^n,
\end{equation}
where $\mathcal{B}$  and each of  the $n^2$ entries of the random
matrix $A$ are i.i.d.~$\mathrm{N}(0,1)$ random variables. Observe that to ensure that the results are invariant under time-scaling, is necessary to add the term $\mathcal{B}$ in the left-hand side of Equation \eqref{sistemadiscret}.
Then,
 given a linear discrete random system \eqref{sistemadiscret}, its
characteristic random polynomial associated with the matrix $\frac1{\mathcal{B}}A$ is
$$
Q(\lambda)=Q_n\lambda^n+Q_{n-1}\lambda^{n-1}+\cdots+Q_1\lambda+Q_0
$$
where each random variable $Q_j$ is a polynomial in the variables
$1/\mathcal{B},A_{1,1},\ldots,A_{n,n}$ which has a complicated distribution function.
We denote by $X$ the random variables that assigns to
each $Q$ its number of roots with modulus smaller than $1,$ that is, 
the stability index of the matrix $\frac1{\mathcal{B}}A.$ Also $p_k$
denotes the probabilities that $X$ takes the value $k.$

As we will see in the examples, in this case
the condition $p_k=p_{n-k}$ is no longer satisfied. Among
other reasons it happens that the entries of  $A^{-1}$
have not simple distributions.   Since we do not know other relations on the
probabilities $p_k$  apart from the trivial one $\sum_{k=0}^n p_k=1$.
Since this is directly fulfilled by the observed relative frequencies, in
this case we do not perform the least squares refinement.

The case $n=1$ is the only one that we have been able to solve
analytically. Notice that in this situation the only solution of
$Q(\lambda)=0$ is $\lambda=A/B,$ with $A$ and $B$ independent and
$\mathrm{N}(0,1)$. Hence $p_0=P(|A/B|>1)$ and $p_1=P(|A/B|<1)=P(|B/A|>1).$
Since $A/B$ and $B/A$ have the same distribution it holds that
$p_0=p_1=1/2.$

As in the other models, for each dimension $n\leq 10,$ we generate
$10^8$ discrete systems of the form \eqref{sistemadiscret}. For each
matrix $\frac1{\mathcal{B}}A$ we have computed the characteristic
polynomial $Q$ and its associated polynomial $Q^\star$ (see
Equation \eqref{e:polyestar}) and have counted the number of roots of this
last polynomial by using the Routh-Hurwitz zero counter. For $n\geq
5$ and in order to decrease the computation time we have directly
numerically computed the eigenvalues of the matrix and counted the
number of them with modulus less than one. The results  obtained are
shown in Table 4.

\bigskip

\begin{minipage}{0.5\textwidth}
\begin{tabular}{|l|l|}
\hline
Dimension & Observed frequency  \\
\hline
\hline
$n=1$ & $\widetilde{p}_0=0.49994$  \\
{}  & $\widetilde{p}_1=0.50006$  \\
\hline
$n=2$ & $\widetilde{p}_0=0.46348$  \\
{}  & $\widetilde{p}_1=0.27705$  \\
{}  & $\widetilde{p}_2=0.25947$  \\
\hline
$n=3$ & $\widetilde{p}_0=0.45261$  \\
{}  & $\widetilde{p}_1=0.25828$   \\
{}  & $\widetilde{p}_2=0.15351$    \\
{}  & $\widetilde{p}_3=0.13560$  \\
\hline
$n=4$ & $\widetilde{p}_0=0.45040$ \\
{}  & $\widetilde{p}_1=0.24732$  \\
{}  & $\widetilde{p}_2=0.14799$ \\
{}  & $\widetilde{p}_3=0.08127$  \\
{}  & $\widetilde{p}_4=0.07302$  \\
\hline
$n=5$ & $\widetilde{p}_0=0.44957$  \\
{}  & $\widetilde{p}_1=0.24536$   \\
{}  & $\widetilde{p}_2=0.13956$  \\
{}  & $\widetilde{p}_3=0.08116$  \\
{}  & $\widetilde{p}_4=0.04443$ \\
{}  & $\widetilde{p}_5=0.03992$  \\
\hline
\end{tabular}

\end{minipage}
\begin{minipage}{0.5\textwidth}
\vspace{-3cm}
\begin{tabular}{|l|l|}
\hline
Dimension & Observed frequency \\
\hline
\hline
$n=6$ & $\widetilde{p}_0=0.44944$ \\
{}  & $\widetilde{p}_1=0.24419$  \\
{}  & $\widetilde{p}_2=0.13838$ \\
{}  & $\widetilde{p}_3=0.07536$ \\
{}  & $\widetilde{p}_4=0.04606$  \\
{}  & $\widetilde{p}_5=0.02449$  \\
{}  & $\widetilde{p}_6=0.02209$ \\
\hline
$n=7$ & $\widetilde{p}_0=0.44937$  \\
{}  & $\widetilde{p}_1=0.24394$  \\
{}  & $\widetilde{p}_2=0.13723$  \\
{}  & $\widetilde{p}_3=0.07480$  \\
{}  & $\widetilde{p}_4=0.04226$  \\
{}  & $\widetilde{p}_5=0.02636$  \\
{}  & $\widetilde{p}_6=0.01367$  \\
{}  & $\widetilde{p}_7=0.01236$  \\
\hline
\end{tabular}

\end{minipage}

\begin{minipage}{0.5\textwidth}
\begin{tabular}{|l|l|}
\hline
Dimension & Observed frequency  \\
\hline
\hline
$n=8$ & $\widetilde{p}_0=0.44937$  \\
{}  & $\widetilde{p}_1=0.24381$  \\
{}  & $\widetilde{p}_2=0.13702$ \\
{}  & $\widetilde{p}_3=0.07388$  \\
{}  & $\widetilde{p}_4=0.04207$  \\
{}  & $\widetilde{p}_5=0.02394$ \\
{}  & $\widetilde{p}_6=0.01526$ \\
{}  & $\widetilde{p}_7=0.00768$  \\
{}  & $\widetilde{p}_8=0.00698$  \\
\hline
$n=9$ & $\widetilde{p}_0=0.44941$  \\
{}  & $\widetilde{p}_1=0.24374$   \\
{}  & $\widetilde{p}_2=0.13680$   \\
{}  & $\widetilde{p}_3=0.07371$    \\
{}  & $\widetilde{p}_4=0.04139$  \\
{}  & $\widetilde{p}_5=0.02400$  \\
{}  & $\widetilde{p}_6=0.01374$  \\
{}  & $\widetilde{p}_7=0.00889$  \\
{}  & $\widetilde{p}_8=0.00434$ \\
{}  & $\widetilde{p}_9=0.00397$  \\
\hline
\end{tabular}
\end{minipage}
\begin{minipage}{0.5\textwidth}
\vspace{-4.8cm}
\begin{tabular}{|l|l|}
\hline
Dimension & Observed frequency \\
\hline
\hline
$n=10$ & $\widetilde{p}_0=0.44934$ \\
{}  & $\widetilde{p}_1=0.24371$  \\
{}  & $\widetilde{p}_2=0.13687$  \\
{}  & $\widetilde{p}_3=0.07358$ \\
{}  & $\widetilde{p}_4=0.04129$  \\
{}  & $\widetilde{p}_5=0.02348$ \\
{}  & $\widetilde{p}_6=0.01388$  \\
{}  & $\widetilde{p}_7=0.00792$ \\
{}  & $\widetilde{p}_8=0.00520$  \\
{}  & $\widetilde{p}_9=0.00247$ \\
{}  & $\widetilde{p}_{10}=0.00226$\\
\hline
\end{tabular}

%
%
\end{minipage}  
  
  \bigskip

\centerline{{\bf Table 4.} Stability indexes for linear
random maps.}

\section{Linear random difference equations of order~$n$}\label{ss:LRdifference}

Finally we consider difference equations of order $n$ of type
$$
A_nx_{k+n}+A_{n-1}x_{k+n-1}+\cdots +A_1 x_{k+1}+A_0x_k=0,
$$
where all
the coefficients are i.i.d.~random
variables with $\mathrm{N}(0,1)$ distribution. In this situation, the stability index is given by the number of
zeros with modulus smaller than 1 of the random characteristic polynomial
$Q(\lambda)=A_n\lambda^{n}+A_{n-1}\lambda^{n-1}+\cdots
+A_1\lambda+A_0.$
As in the preceding sections let  $X$ be the random variable  that counts the number of roots
of $Q(\lambda)$ with  modulus smaller than 1 and set
$p_k=P(X=k)$ for $k=0,1,\ldots,n.$

Before proving some relations among the probabilities $p_k,$ we
give two preliminary lemmas. Let $\operatorname{erf}(x)=\frac 2{\sqrt{\pi}}\int_0^x{\rm
e}^{-u^2}{\rm d} u$ be the error function. The following result is
stated in~\cite{Bri}. We prove it for the sake of completeness.

\begin{lem}\label{le:bri} For $\alpha>0$ and $\beta\in\R,$
\[
F(\alpha,\beta):=\int_0^\infty {\rm
e}^{-\alpha^2x^2}\operatorname{erf}(\beta x)\,{\rm
d}x=\frac{\arctan(\beta/\alpha)}{\alpha\sqrt{\pi}}.
\]
\end{lem}
\begin{proof}
Fixed $\alpha>0,$ the function that defines  $F$ is absolutely
integrable because $|\operatorname{erf}(x)|\le 1.$ Moreover its
partial derivative with respect to $\beta$ is also absolutely
integrable. Hence $\lim_{\beta\to0} F(\alpha,\beta)=F(\alpha,0)=0$
and
\begin{align*}
\frac{\partial F(\alpha,\beta)}{\partial \beta}&=\int_0^\infty
\frac{\partial}{\partial \beta}\left( {\rm
e}^{-\alpha^2x^2}\operatorname{erf}(\beta x)\right)\,{\rm d}x=
\frac2{\sqrt{\pi}}\int_0^\infty x\, {\rm e}^{-\alpha^2x^2} {\rm
e}^{-\beta^2x^2}\,{\rm d}x\\&=\frac2{\sqrt{\pi}}\int_0^\infty x\,
{\rm e}^{-(\alpha^2+\beta^2)\, x^2}\,{\rm
d}x=\frac1{(\alpha^2+\beta^2)\sqrt{\pi}}.
\end{align*}
Therefore
\[
F(\alpha,\beta)=F(\alpha,0)+\int_0^\beta \frac{\partial
F(\alpha,t)}{\partial t}\,{\rm d}t=\int_0^\beta
\frac1{(\alpha^2+t^2)\sqrt{\pi}}\,{\rm
d}t=\frac{\arctan(\beta/\alpha)}{\alpha\sqrt{\pi}},
\]
as we wanted to prove.
\end{proof}

Next result is a consequence of the previous lemma.

\begin{lem}\label{le:normals}
Let $U\sim \mathrm{N}(0,\sigma^2)$ and $V\sim \mathrm{N}(0,\rho^2)$  be independent
normal random variables. Then $P(U^2-V^2>0)
=\frac2\pi\arctan(\sigma/\rho).$
\end{lem}

\begin{proof} The joint density function of the
random vector $(U,V)$ is $f_\sigma(u)f_\rho(v),$ where $f_s(u)={\rm
e}^{-u^2/(2s^2)}/(\sqrt{2\pi}s)$. Observe that the points $(u,v)\in\mathbb{R}^2$ such that  $u^2-v^2>0$ is the region where $-|u|<v<|u|$, hence  by symmetry,
\begin{align*}
P(U^2-V^2>0)&=4\int_0^\infty f_\sigma(u)\int_0^u f_\rho(v)\,{\rm
d}v\,{\rm d}u=\frac4{2\pi \sigma\rho}\int_0^\infty {\rm
e}^{-u^2/(2\sigma^2)}\int_0^u {\rm e}^{-v^2/(2\rho^2)}\,{\rm
d}v\,{\rm d}u\\&=\frac2{\pi \sigma\rho}\int_0^\infty {\rm
e}^{-u^2/(2\sigma^2)}\operatorname{erf}\Big(\frac
u{\sqrt{2}\rho}\Big)\sqrt{\frac\pi 2}\rho\,{\rm d}u=\frac
2\pi\arctan\Big(\frac\sigma\rho\Big),
\end{align*}
where in the last equality we have used Lemma \ref{le:bri}.
\end{proof}

Notice that with the notation of the  above lemma,
$P(U^2-V^2>0)+P(U^2-V^2<0)=1.$ Hence
\begin{equation}\label{eq:comple}
P(U^2-V^2<0)=1-\frac 2\pi\arctan\Big(\frac\sigma\rho\Big)=\frac
2\pi\arctan\Big(\frac\rho\sigma\Big),
\end{equation}
where we have used  the fact that $\arctan(x)+\arctan(1/x)=\pi/2$ or, simply,
the same lemma interchanging $U$ and $V.$ Observe also that when
$\sigma=\rho,$ $P(U^2-V^2>0)=P(U^2-V^2<0)=1/2,$ a result that, in
fact, is a straightforward consequence that in this situation
$U^2-V^2$ and $V^2-U^2$ have the same distribution.

\begin{propo}\label{p:propoEDiff} With the above notation:
\begin{enumerate}
\item[(a)] $\sum_{k=0}^n p_k=1.$
\item[(b)] For all $k\in\{0,1,\ldots,n\},$ $p_k=p_{n-k}.$
\item[(c)] When $n$ is odd, $\sum_{i\,\,\mathrm{even}} p_i=\sum_{i\,\,\mathrm{odd}} p_i=\frac{1}{2}.$
\item[(d)] When $n=2k$ is even,
\begin{equation}\label{e:sumaarctan}
\sum_{i\,\,\mathrm{even}} p_i=\frac
2\pi\arctan\Big(\sqrt{\frac{k+1}k}\,\Big)\quad \mbox{and}\quad
\sum_{i\,\,\mathrm{odd}} p_i=\frac
2\pi\arctan\Big(\sqrt{\frac{k}{k+1}}\,\Big).
\end{equation}

\end{enumerate}
\end{propo}

\begin{proof}
The first assertion is obvious. To see the second one we compare the
difference equation
$a_nx_{k+n}+a_{n-1}x_{k+n-1}+\cdots+a_1x_{k+1}+a_0x_k=0,$ with
$a_i\in\R,\,i=0,1,\ldots,n,$ with characteristic polynomial
$Q(\lambda)=a_n\lambda^{n}+a_{n-1}\lambda^{n-1}+\cdots+a_2\lambda^2+a_1\lambda+a_0$
with the difference equation $a_nx_k+a_{n-1}x_{k+1}+\cdots
+a_{0}x_{k+n}=0$ with characteristic polynomial $
{Q}^*(\lambda)=a_n+a_{n-1}\lambda+\cdots
+a_1\lambda^{n-1}+a_0\lambda^n.$ Notice that if $Q(\lambda)$ has $m$
non-zero roots with modulus smaller than 1 and $n-m$ with modulus
bigger than 1, then the converse follows for $Q^*(\lambda)$
because $Q(\lambda)=0$ if and only if
$Q^*(\frac{1}{\lambda})=0.$
 From this result applied to the corresponding random polynomials
we get that $p_k=p_{n-k},$ because both have identically distributed
coefficients. So we have proved statement (b). To prove items $(c)$ and $(d)$ recall first that it was proved in item $(c)$
of Proposition~\ref{p:propoEDOL} that a polynomial $Q(\lambda)=a_n\lambda^{n}+a_{n-1}\lambda^{n-1}+\cdots+a_2\lambda^2+a_1\lambda+a_0$, without roots with zero real part,  has an even number of roots with negative real part if and only if $a_na_0>0.$ By using the  polynomial
\begin{align*}
Q^\star(z)&=a_n(z+1)^n+a_{n-1}(z+1)^{n-1}(z-1)+\ldots+a_0(z-1)^n\\&=(a_n+a_{n-1}+\cdots+a_1+a_0)z^n+\cdots+(a_n-a_{n-1}+a_{n-2}-\cdots +(-1)^na_0),
\end{align*}
introduced in Section~\ref{s:linear} (Equation \eqref{e:polyestar}) we get that $Q(\lambda),$ without roots of modulus 1, has an even number ($2m$) of roots with modulus smaller than 1 if and only if $Q^\star(z)$ has  exactly $2m$ roots with negative real part and this happens if and only if $ (a_n+a_{n-1}+\cdots+a_1+a_0)\cdot(a_n-a_{n-1}+a_{n-2}-\cdots +(-1)^na_0)>0.$ Hence, considering the corresponding random polynomials, we have that
\begin{align*}
\sum_{i\,\,\mathrm{even}} p_i&=
    P\big( (A_n+A_{n-1}+\cdots+A_0)\cdot(A_n-A_{n-1}+\cdots +(-1)^nA_0)>0\big)\\&=P(U^2-V^2>0),
\end{align*}
 where $U=A_n+A_{n-2}+A_{n-4}+\cdots$ and $V=A_{n-1}+A_{n-3}+A_{n-5}+\cdots$ and the sums end either at $A_0$ or $A_1$ according  the parity of $n.$ Since $A_j\sim \mathrm{N}(0,1)$ and all $A_j$ are independent we get that when $n=2k$  (resp. $n=2k-1$) then  $U\sim \mathrm{N}(0,k+1)$ (resp.  $U\sim \mathrm{N}(0.k)$)  and $V\sim \mathrm{N}(0,k)$ and $U$ and $V$ are independent. Hence, by using Lemma~\ref{le:normals}, we obtain that when $n=2k-1,$ $P(U^2-V^2>0)=1/2$ and that when $n=2k,$
 \[
 \sum_{i\,\,\mathrm{even}} p_i=P(U^2-V^2>0)=\frac 2\pi\arctan\Big(\sqrt{\frac{k+1}k}\,\Big).
 \]
The sum of all $p_i$ when $i$ is odd  can be  obtained from the above formula, see also \eqref{eq:comple}.
 \end{proof}

\begin{corol}\label{c:propoEDD}
(i) Consider the linear random homogeneous difference equation of order $n,$ let $X$ be the random variable
defined above and  $p_k=P(X=k)$. Then the probabilities $p_k$
satisfy all the conclusions of Lemma \ref{l:lemmageneral}. In particular $E(X)=n/2.$

(ii)  Moreover the new affine relations given in Equations \eqref{e:sumaarctan} hold. In particular, for $n=2,$ $p_0=p_2=\frac1\pi\arctan(\sqrt{2})$ and
$p_1=\frac2\pi\arctan(\frac1{\sqrt{2}})$; and for $n=4,$  $ p_1=p_3=\frac1\pi\arctan(\sqrt{\frac{2}{3}}).$
\end{corol}

In this case, and for the situations where we have not been able to obtain the exact probabilities  we have done similar computations than in the previous section, first with the Monte Carlo method, generating for each order $n=0,\ldots, 10,$
$10^8$ random vectors $(A_0,\ldots,A_n)\in \R^{n+1}$  whose
components are pseudo-random numbers with $\mathrm{N}(0,1)$ distribution. Then,  by using the relations in Proposition~\ref{p:propoEDiff} and Corollary
\ref{c:propoEDD} we have  performed a least squares refinement.

    For
    instance for $n=4$, by Corollary \ref{c:propoEDD} we have
    $p_1=p_3=\arctan(\sqrt{2/3})/\pi\simeq 0.217953$;
    $p_2=2\arctan(\sqrt{3/2})/\pi-2p_0$ and $p_4=p_0$. Hence, we fix the values $\widehat{p}_1=p_1$ and $\widehat{p}_3=p_3$ and  system \eqref{e:pMqb}  can be written in the form
    $$\mathrm{M}\widehat{\mathbf{q}}+b=
    \left(\begin{array}{r}
    1\\
    -2\\
    1
    \end{array}\right)\cdot\left(\widehat{p}_0\right) +\left(\begin{array}{c}
    0\\
    \frac{2}{\pi}\arctan\left(\sqrt{\frac{3}{2}}\right)\\
    0
    \end{array}\right)=\left(\begin{array}{c}
    \widetilde{p}_0\\
    \widetilde{p}_2\\
    \widetilde{p}_4
    \end{array}\right).$$
    Hence we can easily find the least squares solution of the above incompatible linear system:
    $$
    (1,-2,1)\cdot\left[\left(\begin{array}{r}
    1\\
    -2\\
    1
    \end{array}\right)\cdot\left(\widehat{p}_0\right) +\left(\begin{array}{c}
    0\\
    \frac{2}{\pi}\arctan\left(\sqrt{\frac{3}{2}}\right)\\
    0
    \end{array}\right)\right]=(1,-2,1)\cdot\left(\begin{array}{c}
    \widetilde{p}_0\\
    \widetilde{p}_2\\
    \widetilde{p}_4
    \end{array}\right),
    $$
    and thus we get the result in Equation \eqref{e:leastquaresolution}:
    $
    \widehat{p}_0=\frac{1}{6}\left(\widetilde{p}_0-2\widetilde{p}_2
    +\widetilde{p}_4\right)+\frac{2}{3\pi}\arctan\left(\sqrt{{3}/{2}}\right),
    $
    and therefore
       $\widehat{p}_4=\widehat{p}_0$ and 
    $\widehat{p}_2= 2\arctan(\sqrt{3/2})/\pi-2\widehat{p}_0. $
    Since our Monte Carlo simulations give
    $$
    (\widetilde{p}_0,\widetilde{p}_2,\widetilde{p}_4)= \left(\frac{2056203}{20000000},
    \frac{7169499}{20000000},
    \frac{10285619}{100000000} \right)\simeq\left(0.10281,0.35847,0.10286\right),
    $$ the above relations show that
    $
    (\widehat{p}_0,\widehat{p}_2,\widehat{p}_4) \simeq\left(0.10282,0.35846,0.10282\right).
    $

    All our  results are collected in Table 5.

    \begin{center}
        \begin{tabular}{|l|l|l|l|}
            \hline
            Order & Observed frequency &  Least squares & Relations (Prop.~\ref{p:propoEDiff} and Cor.~\ref{c:propoEDD}))\\
            \hline
            \hline
            $n=1$ & $\widetilde{p}_0=0.49991$ & &$p_0=0.5$\\
            {}  & $\widetilde{p}_1=0.50009$ &   &$p_1=0.5$\\
            \hline
            $n=2$ & $\widetilde{p}_0=0.30410$ &&$p_0=\frac1\pi\arctan(\sqrt{2})\simeq0.304087$\\
            {}  & $\widetilde{p}_1=0.39184$ &   &$p_1=\frac2\pi\arctan(\frac1{\sqrt{2}})\simeq0.391826$\\
            {}  & $\widetilde{p}_2=0.30406$ &  &$p_2=\frac1\pi\arctan(\sqrt{2})\simeq0.304087$\\
            \hline
            $n=3$ & $\widetilde{p}_0=0.17251$ &$\widehat{p}_0=0.17248$ &$p_0$\\
            {}  & $\widetilde{p}_1=0.32752$ &  $\widehat{p}_1=0.32752$ &$p_1=\frac{1}{2}-p_0$\\
            {}  & $\widetilde{p}_2=0.32753$ &  $\widehat{p}_2=0.32752$ &$p_2=\frac{1}{2}-p_0$\\
            {}  & $\widetilde{p}_3=0.17244$ &  $\widehat{p}_3=0.17248$&$p_3=p_0$\\
            \hline
            $n=4$ & $\widetilde{p}_0=0.10281$ &  $\widehat{p}_0=0.10282$ &$p_0$\\
            {}  & $\widetilde{p}_1=0.21792$ &  $\widehat{p}_1=0.21795$ &$p_1=\frac1\pi\arctan(\sqrt{\frac{2}{3}})\simeq0.217953$\\
            {}  & $\widetilde{p}_2=0.35847$ &  $\widehat{p}_2=0.35846$ &$p_2=\frac2\pi\arctan(\sqrt{\frac{3}{2}})-2p_0$\\
            {}  & $\widetilde{p}_3=0.21794$ &  $\widehat{p}_3=0.21795$ &$p_3=\frac1\pi\arctan(\sqrt{\frac{2}{3}})\simeq0.217953$\\
            {}  & $\widetilde{p}_4=0.10286$ &  $\widehat{p}_4=0.10282$ &$p_4=p_0$\\
            \hline
            $n=5$ & $\widetilde{p}_0=0.05909$ & $\widehat{p}_0=0.05909$ &  $p_0$\\
            {}  & $\widetilde{p}_1=0.15331$ & $\widehat{p}_1=0.15333$ &  $p_1$\\
            {}  & $\widetilde{p}_2=0.28760$ & $\widehat{p}_2=0.28758$ &  $p_2=\frac{1}{2}-p_0-p_1$\\
            {}  & $\widetilde{p}_3=0.28756$ & $\widehat{p}_3=0.28758$ &  $p_3=\frac{1}{2}-p_0-p_1$\\
            {}  & $\widetilde{p}_4=0.15335$ & $\widehat{p}_4=0.15333$ &  $p_4=p_1$\\
            {}  & $\widetilde{p}_5=0.05908$ & $\widehat{p}_5=0.05909$ &  $p_5=p_0$\\
            \hline
            $n=6$ & $\widetilde{p}_0=0.03501$ & $\widehat{p}_0=0.03502$ &  $p_0$\\
            {}  & $\widetilde{p}_1=0.09726$ & $\widehat{p}_1=0.09726$ &  $p_1$\\
            {}  & $\widetilde{p}_2=0.23777$ & $\widehat{p}_2=0.23779$ &  $p_2=\frac1\pi\arctan(\sqrt{\frac{4}{3}})-p_0$\\
            {}  & $\widetilde{p}_3=0.25985$ & $\widehat{p}_3=0.25986$ &  $p_3=\frac2\pi\arctan(\sqrt{\frac{3}{4}})-2p_1$\\
            {}  & $\widetilde{p}_4=0.23781$ & $\widehat{p}_4=0.23779$ &  $p_4=\frac1\pi\arctan(\sqrt{\frac{4}{3}})-p_0$\\
            {}  & $\widetilde{p}_5=0.09724$ & $\widehat{p}_5=0.09726$ &  $p_5=p_1$\\
            {}  & $\widetilde{p}_6=0.03505$ & $\widehat{p}_6=0.03502$ &  $p_6=p_0$\\
            \hline
            
                    \end{tabular}
    \end{center}

\vspace{-1cm}
    \begin{center}
        \begin{tabular}{|l|l|l|l|}
            \hline
            Order & Observed frequency &  Least squares &  Relations (Prop.~\ref{p:propoEDiff} and Cor.~\ref{c:propoEDD})\\
            \hline
            \hline
            $n=7$ & $\widetilde{p}_0=0.02025$ & $\widehat{p}_0=0.02025$ &  $p_0$\\
            {}  & $\widetilde{p}_1=0.06432$ & $\widehat{p}_1=0.06430$ &  $p_1$\\
            {}  & $\widetilde{p}_2=0.17174$ & $\widehat{p}_2=0.17176$ &  $p_2$\\
            {}  & $\widetilde{p}_3=0.24376$ & $\widehat{p}_3=0.24369$ &  $p_3=\frac{1}{2}-p_0-p_1-p_2$\\
            {}  & $\widetilde{p}_4=0.24361$ & $\widehat{p}_4=0.24369$ &  $p_4=\frac{1}{2}-p_0-p_1-p_2$\\
            {}  & $\widetilde{p}_5=0.17177$ & $\widehat{p}_5=0.17176$ &  $p_5=p_2$\\
            {}  & $\widetilde{p}_6=0.06428$ & $\widehat{p}_6=0.06430$ &  $p_6=p_1$\\
            {}  & $\widetilde{p}_7=0.02025$ & $\widehat{p}_7=0.02025$ &  $p_7=p_0$\\
            \hline
            $n=8$ & $\widetilde{p}_0=0.01194$ & $\widehat{p}_0=0.01196$ &  $p_0$\\
            {}  & $\widetilde{p}_1=0.03994$ & $\widehat{p}_1=0.03994$ &  $p_1$\\
            {}  & $\widetilde{p}_2=0.12272$ & $\widehat{p}_2=0.12726$ &  $p_2$\\
            {}  & $\widetilde{p}_3=0.19230$ & $\widehat{p}_3=0.19234$ &  $p_3=\frac1\pi\arctan(\sqrt{\frac{4}{5}})-p_1$\\
            {}  & $\widetilde{p}_4=0.25701$ & $\widehat{p}_4=0.25700$ &  $p_4=\frac2\pi\arctan(\sqrt{\frac{5}{4}})-2p_0-2p_2$\\
            {}  & $\widetilde{p}_5=0.19238$ & $\widehat{p}_5=0.19234$ &  $p_5=\frac1\pi\arctan(\sqrt{\frac{4}{5}})-p_1$\\
            {}  & $\widetilde{p}_6=0.12724$ & $\widehat{p}_6=0.12726$ &  $p_6=p_2$\\
            {}  & $\widetilde{p}_7=0.03994$ & $\widehat{p}_7=0.03994$ &  $p_7=p_1$\\
            {}  & $\widetilde{p}_8=0.01197$ & $\widehat{p}_8=0.01196$ &  $p_8=p_0$\\
            \hline
            $n=9$ & $\widetilde{p}_0=0.00693$ & $\widehat{p}_0=0.00693$ &  $p_0$\\
            {}  & $\widetilde{p}_1=0.02556$ & $\widehat{p}_1=0.02556$ &  $p_1$\\
            {}  & $\widetilde{p}_2=0.08711$ & $\widehat{p}_2=0.08711$ &  $p_2$\\
            {}  & $\widetilde{p}_3=0.15653$ & $\widehat{p}_3=0.15653$ &  $p_3$\\
            {}  & $\widetilde{p}_4=0.22389$ & $\widehat{p}_4=0.22386$ &  $p_4=\frac{1}{2}-p_0-p_1-p_2-p_3$\\
            {}  & $\widetilde{p}_5=0.22382$ & $\widehat{p}_5=0.22386$ &  $p_5=\frac{1}{2}-p_0-p_1-p_2-p_3$\\
            {}  & $\widetilde{p}_6=0.15654$ & $\widehat{p}_6=0.15653$ &  $p_6=p_3$\\
            {}  & $\widetilde{p}_7=0.08712$ & $\widehat{p}_7=0.08711$ &  $p_7=p_2$\\
            {}  & $\widetilde{p}_8=0.02557$ & $\widehat{p}_8=0.02556$ &  $p_8=p_1$\\
            {}  & $\widetilde{p}_9=0.00693$ & $\widehat{p}_9=0.00693$ &  $p_9=p_0$\\
            \hline
            
            \end{tabular}
    \end{center}

    \begin{center}
        \vspace{-1cm}
        \begin{tabular}{|l|l|l|l|}
            \hline
            Order & Observed frequency &  Least squares &  Relations (Prop.~\ref{p:propoEDiff} and Cor.~\ref{c:propoEDD})\\
            \hline
   \hline
            $n=10$ & $\widetilde{p}_0=0.00409$ & $\widehat{p}_0=0.00411$ &  $p_0$\\
            {}  & $\widetilde{p}_1=0.01567$ & $\widehat{p}_1=0.01566$ &  $p_1$\\
            {}  & $\widetilde{p}_2=0.06089$ & $\widehat{p}_2=0.06091$ &  $p_2$\\
            {}  & $\widetilde{p}_3=0.11500$ & $\widehat{p}_3=0.11497$ &  $p_3$\\
            {}  & $\widetilde{p}_4=0.19950$ & $\widehat{p}_4=0.19947$ &  $p_4=\frac1\pi\arctan(\sqrt{\frac{6}{5}})-p_0-p_2$\\
            {}  & $\widetilde{p}_5=0.20978$ & $\widehat{p}_5=0.20976$ &  $p_5=\frac2\pi\arctan(\sqrt{\frac{5}{6}})-2p_1-2p_3$\\
            {}  & $\widetilde{p}_6=0.19941$ & $\widehat{p}_6=0.19947$ &  $p_6=\frac1\pi\arctan(\sqrt{\frac{6}{5}})-p_0-p_2$\\
            {}  & $\widetilde{p}_7=0.11499$ & $\widehat{p}_7=0.11497$ &  $p_7=p_3$\\
            {}  & $\widetilde{p}_8=0.06088$ & $\widehat{p}_8=0.06091$ &  $p_8=p_2$\\
            {}  & $\widetilde{p}_9=0.01570$ & $\widehat{p}_9=0.01566$ &  $p_9=p_1$\\
            {}  & $\widetilde{p}_{10}=0.00408$ & $\widehat{p}_{10}=0.00411$ &  $p_{10}=p_0$\\
            \hline
        \end{tabular}

        \bigskip

        {\bf Table 5.} Stability indexes for order $n$ linear random
        homogeneous difference equations.
    \end{center}

\section*{Acknowledgments}    
We want to thank Professors Maria Jolis and
Joan Torregrosa for their helpful indications.
We also acknowledge the anonymous reviewer for the careful reading and for addressing us very interesting suggestions that have allowed us to improve our work.


\begin{thebibliography}{12}
\bibitem{AL} J.C.~Art\'es, J.~Llibre. \textsl{Statistical measure of quadratic vector fields.}
Resenhas 6 (2003), 85--97.

\bibitem{AG} S.~Asmussen, P.W.~Glynn.
Stochastic simulation: algorithms and analysis. Springer, New York,~2007.

\bibitem{BFS} P.~Bratley, B.L.~Fox, L. E.~Schrage.
A guide to simulation, 2nd ed. Springer, New York,~1987.

\bibitem{Bri}  K.~Briggs. \textsl{Integrals involving erf.}  \url{http://keithbriggs.info/documents/erf-integrals.pdf}
(Accessed November 23, 2020).



\bibitem{CS10}
B.M.~Chen-Charpentier, D.~Stanescu. \textsl{Epidemic models with random coefficients.}
Math. Comput. Modelling 52 (2010), 1004--1010.

\bibitem{CS13}
B.M.~Chen-Charpentier, D.~Stanescu.
\textsl{Virus propagation with randomness.} Math. Comput. Modelling 57 (2013), 1816--1821

\bibitem{CGM20} A.~Cima, A.~Gasull, V.~Ma\~{n}osa. \textsl{Phase portraits of random planar homogeneous vector fields.} \href{https://arxiv.org/abs/1911.11441}{arXiv:1911.11441 [math.DS]}. To appear in Qual. Theory Dyn. Syst. 



\bibitem{Conrad} K.~Conrad. \textsl{Probability distributions and maximum entropy}.\\
\url{http://www.math.uconn.edu/~kconrad/blurbs/analysis/entropypost.pdf} (Accessed November 23, 2020).


\bibitem{DB} G.~Dahlquist, \r{A}.~Bj\"orck. Numerical methods. Prentice-Hall, Englewood
Cliffs NJ, 1974. Section 5.7 pp. 196--201.


\bibitem{GR} I.S.~Gradshteyn, I.M.~Ryzhik, \textsl{Routh-Hurwitz Theorem.} \S 15.715 in
Tables of Integrals, Series, and Products, 6th ed. Academic Press, San Diego CA, 2000,
p.~1076.

\bibitem{GS} M.~Grigoriu, T.~Soong. Random vibration in mechanical and structural systems.
Prentice Hall, New Jersey, 1993.

\bibitem{Jones} F.B.~Jones. \textsl{Connected and disconnected plane sets and the functional
equation $f(x)+f(y)=f(x+y)$.} Bull. Amer. Math. Soc. 48 (1942), 115--120.

\bibitem{L} C.~Lemieux. Monte Carlo and quasi-Monte Carlo sampling.
Springer, New York, 2009.

\bibitem{LPP} H.~Lopes, B.~Pagnoncelli, C.~Palmeira.
\textsl{Coeficientes aleat\'{o}rios de equa\c{c}\~{o}es diferenciais ordin\'{a}rias
lineares.} Matem\'{a}tica Universit\'{a}ria 14  (2008), 44--50. Translated
to English at \url{http://bernardokp.uai.cl/preprint.pdf} with title \textsl{Random
linear systems and simulation} (Accessed November 23, 2020).



\bibitem{Mars} G.~Marsaglia. \textsl{Choosing a point from the surface of a sphere.}
The Annals of Mathematical Statistics 43 (1972), 645--646.

\bibitem{MT}   G.~Marsaglia, W.~Tsang.  \textsl{The Ziggurat method for generating random variables.} Journal of Statistical Software 5(8) (2000), 1--7.

\bibitem{MN} M.~Matsumoto, T.~Nishimura. \textsl{Mersenne twister: a 623-dimensionally equidistributed uniform pseudo-random number generator}. ACM Trans. Model. Comput. Simul. 8, 1 (1998), 3--30.

\bibitem{Morgan} B.J.T.~Morgan. Applied stochastic modelling. Arnold Publishers, London, 2000.

\bibitem{Mull} M.E.~Muller. \textsl{A note on a method for generating points uniformly on
$N$-dimensional spheres.} Communications of the ACM 2 (1959), 19--20.



\bibitem{Schott} J.R.~Schott. Matrix analysis for statistics. John Wiley \& Sons, New
York, 1997.

\bibitem{SC09} D.~Stanescu, B.M.~Chen-Charpentier. \textsl{Random coefficient differential
equation models for bacterial growth}. Math. Comput. Modelling 50 (2009),  885--895.

\bibitem{SB} J.~Stoer, R.~Bulirsch. Introduction to numerical analysis, 3rd ed. Springer, New York, 2002.

\bibitem{Strog} S.H.~Strogatz. Nonlinear dynamics and chaos. Westview Press, Cambridge MA, 1994.



\end{thebibliography}
\end{document}